\documentclass{amsart}[11pt,A4]
\usepackage[margin=2.54cm]{geometry}
\usepackage[T1]{fontenc}
\usepackage[latin9]{inputenc}
\usepackage{amsthm,amsmath,amssymb,color,graphicx,url}
\usepackage{float}
\usepackage{tikz}
\usepackage{graphicx}
\usepackage{mathtools}
\usepackage{makecell}

\usetikzlibrary{calc}

\newtheorem{theorem}{Theorem}
\newtheorem{lemma}[theorem]{Lemma}
\newtheorem{corollary}[theorem]{Corollary}
\newtheorem{question}[theorem]{Question}
\newtheorem{proposition}[theorem]{Proposition}
\theoremstyle{definition}

\newtheorem{remark}[theorem]{Remark}

\newtheorem{conjecture}[theorem]{Conjecture}

\newcommand{\primozcomment}[1] {\color{magenta} ***  #1 *** \color{black}}

\newcommand{\meo}{{\rm meo}}
\newcommand{\mso}{{\meo_\circ}}

\newcommand{\Aut}{\hbox{\rm Aut}}

\newcommand{\Sym}{{\rm Sym}}
\newcommand{\Alt}{{\rm Alt}}

\newcommand{\la}{\langle}
\newcommand{\ra}{\rangle}

\def\Zent#1{{\bf Z}({{#1}})}

\newcommand{\NN}{\mathbb{N}}
\newcommand{\ZZ}{\mathbb{Z}}

\newcommand{\lcm}{\hbox{lcm}}

\newcommand{\id}{\hbox{id}}

\newcommand{\dPX}{{\vec{\rm PX}}}
\newcommand{\SPX}{{\rm SPX}}

\newcommand{\Z}{{\rm Z}}
\newcommand{\D}{{\rm D}}
\newcommand{\V}{{\rm V}}
\newcommand{\E}{{\rm E}}
\newcommand{\C}{{\rm C}}
\newcommand{\A}{{\rm A}}

\renewcommand{\lcm}{{\rm lcm}}

\makeatletter
\numberwithin{equation}{section}
\numberwithin{figure}{section}
\numberwithin{theorem}{section}

\makeatother

\numberwithin{equation}{section}
\numberwithin{figure}{section}

\title[Orders of automorphisms]{On orders of automorphisms of vertex-transitive graphs}
\author{Primo\v{z} Poto\v{c}nik}
\author{Micael Toledo}
\author{Gabriel Verret}

\address{Primo\v{z} Poto\v{c}nik, Faculty of Mathematics and Physics, University of Ljubljana, Jadranska 21, SI-1000 Ljubljana, Slovenia.\newline
\indent Also affiliated with: Institute of Mathematics, Physics and Mechanics, Jadranska 19, SI-1000 Ljubljana, Slovenia
}
\email{primoz.potocnik@fmf.uni-lj.si}

\address{Micael Toledo, Faculty of Mathematics and Physics, University of Ljubljana, Jadranska 21, SI-1000 Ljubljana, Slovenia}
\email{micael.toledo@imfm.uni-lj.si}

\address{Gabriel Verret, Department of Mathematics, University of Auckland, Private Bag 92019, Auckland 1142, New Zealand}
\email{g.verret@auckland.ac.nz}

\thanks{The authors gratefully acknowledge support of the Slovenian Research Agency: Core Programme P1-0294, Research Project J1-1691 and the Young Researcher Scholarship programme.}

\begin{document}

\begin{abstract}
In this paper we investigate orders, longest cycles and the number of cycles of automorphisms of finite vertex-transitive graphs. In particular, we show that the order of every automorphism of a connected vertex-transitive graph with $n$ vertices and of valence $d$, $d\le 4$, is at most $c_d n$ where $c_3=1$ and $c_4 = 9$. Whether such a constant $c_d$ exists for valencies larger than $4$ remains an unanswered question. Further, we prove that every automorphism $g$ of a finite connected $3$-valent vertex-transitive graph $\Gamma$, $\Gamma \not\cong K_{3,3}$, has a regular orbit, that is, an orbit of $\langle g \rangle$ of length equal to the order of $g$. Moreover, we prove
that in this case either $\Gamma$ belongs to a well understood family of exceptional graphs or at least $5/12$ of the vertices of $\Gamma$ belong to a regular orbit of $g$. Finally, we give an upper bound on the number of orbits of a cyclic group of automorphisms $C$ of a connected $3$-valent vertex-transitive graph $\Gamma$ in terms of the number of vertices of $\Gamma$ and the length of a longest orbit of $C$.
\end{abstract}

\maketitle

\section{Introduction}
\label{sec:intro}

For a permutation $g$ of a finite non-empty set $\Omega$, let
\begin{eqnarray*}
o(g) &= & \hbox { the order of } g; \\
\ell(g) &= & \hbox { the length of the longest orbit of } \langle g \rangle; \\
\mu(g) &= & \hbox { the number of orbits of } \langle g \rangle.
\end{eqnarray*}
Equivalently, if the permutation $g$ is written as a product of disjoint cycles, then $\mu(g)$ equals the number of the cycles (including those of length $1$), $\ell(g)$ represents the length of the longest of these cycles, while $o(g)$ equals the least common multiple of the lengths
of all the cycles of $g$.
Further, for a permutation group $G\le \Sym(\Omega)$ we let
$$
\meo(G)  =  \max_{g\in G} o(g),\quad
\ell(G)  =  \max_{g \in G} \ell(g), \quad
\mu(G)  =  \min_{g \in G} \mu(g), \quad
\mso(G)  =  \max_{\omega\in \Omega} \meo(G_\omega).
$$
where $G_\omega$ denotes the stabiliser of $\omega\in \Omega$ in $G$.
These parameters are of course not mutually independent (see Lemma~\ref{lem:basic1} and Lemma~\ref{lem:mu}).
%

Studying maximal orders of permutations and the longest orbits of cyclic groups
have received considerable attention in the theory of  finite groups, especially in the context of primitive permutation groups \cite{GueMorPreSpi15,GueMorPreSpi16,GuestSpiga}. 
In this paper we propose the investigation of the above parameters in a graph theoretical setting. In particular, for a finite graph $\Gamma$ with the automorphism group $\Aut(\Gamma)$, we are interested in the following invariants:
\begin{center}
\begin{tabular}{rclcl}
$\meo(\Gamma)$ & $=$ & $\meo(\Aut(\Gamma))$ &$\ldots$ & 
   {\em maximal order of an automorphism}; \\
$\ell(\Gamma)$ & $=$ & $\ell(\Aut(\Gamma))$ &$\ldots$ & 
  {\em maximal length of an orbit of an automorphism}; \\ 
$\mu(\Gamma)$ & $=$ & $\mu(\Aut(\Gamma))$ &$\ldots$ & 
  {\em minimal number of orbits of a non-trivial automorphism}; \\ 
$\mso(\Gamma)$ & $=$ & $\mso(\Aut(\Gamma))$ & $\ldots$&
   {\em maximal order of an automorphism fixing a vertex}.\\
\end{tabular}
\end{center}

There is a number of interesting problems related to these invariants, especially
in the context of connected vertex-transitive graphs. Let us mention a few.

The first question is related to the existence and possible classification of graphs admitting
an automorphism with a large order, that is, graphs $\Gamma$ with a large value of 
$\meo(\Gamma)$. 

A classical result of Landau \cite{Landau} states that $\meo(\Sym(n))  = e^{(1+o(1))(n\log n)^{1/2}}$ giving a sub-exponential upper bound on $\meo(\Gamma)$ in terms of the number of vertices $\Gamma$, which is met by the complete graphs.
A more interesting question to ask is whether a better, possibly linear bound holds for connected graphs of fixed valence. In this paper, we consider this question in the context of finite connected vertex-transitive graphs. 

\begin{question}
\label{q:meo}
For which positive integers $d$ does there exist a constant $c_d$ such that
every connected $d$-valent vertex-transitive graph $\Gamma$ with $n$ vertices satisfies
$\meo(\Gamma) \le c_d n$?
\end{question}

In this paper, we answer this question for valencies $3$ and $4$. In particular, we prove the following:

\begin{theorem}
\label{the:34}
Let $\Gamma$ be a finite connected vertex-transitive graph of valence $d \in \{3,4\}$ and let $n$ be the number of vertices of $\Gamma$. If $d=3$, then $\meo(\Gamma) \le n$, and
if $d=4$, then $\meo(\Gamma) \le 9n$.
\end{theorem}

While the constant $c_3=1$ is sharp, as witnessed by the circulants, we have no examples
of vertex-transitive graphs $\Gamma$ of valence $4$ admitting an automorphism of order $9n$. The smallest possible value of the constant $c_4$ might thus be as small as $1$.

Furthermore,  a variant of the Thompson-Wielandt theorem \cite[Corollary 3]{SpigaTW} 
together with an easy arithmetic argument allowed us to deduce
the following fact about arc-transitive locally semi-primitive graphs (recall that
$\Gamma$ is said to be $G$-arc-transitive if $G\le \Aut(\Gamma)$ acts
transitively on the set of ordered pairs of adjacent vertices, called arcs; 
for the definition of local semi-primitivity of graphs see Section~\ref{sec:TW}).

\begin{theorem}
\label{the:WieThomp}
For every positive integer $d$ there exists a constant $c_d$ such that
every connected $G$-arc-transitive $G$-locally-semiprimitive graph $\Gamma$ with $n$ vertices satisfies $\meo(G) \le c_d n$.
\end{theorem}

While the mere existence of the constant $c_d$ as in Question~\ref{q:meo} often suffices for
theoretical applications, funding as small a constant $c_d$ as possible is very desirable for practical applications. What is more, it is often helpful to investigate structural properties inferred by existence of an automorphism $g$ of a large order. The following
result, which was recently used to prove a classification \cite{longorbit} of cubic vertex-transitive graphs $\Gamma$ with $\meo(\Gamma) \ge |\V(\Gamma)|/3$,
gives an explicit upper bound on the number of orbits that an automorphism of
relatively large order has. The term {\em cubic graph} is used here to refer to
a finite connected graph every vertex of which has valence $3$. We will use this term
 throughout the paper. Moreover, $\V(\Gamma)$ is used to denote the vertex-set of a graph $\Gamma$.

\begin{theorem}
\label{the:mainorbits}
Let $\Gamma$ be a cubic vertex-transitive graph with $n$ vertices and let $g \in \Aut(\Gamma)$. Then $\mu(g) \le \frac{17n}{6\, o(g)}$. In particular, $\mu(\Gamma) \le \frac{17n}{6\, \meo(\Gamma)}$.
\end{theorem}

Since clearly $\meo(G) \le  \mso(G) n \le |G_\omega| n$ holds for every permutation group $G$ acting on a set of size $n$, one way of bounding the parameter $\meo(\Gamma)$ by a linear function of $|\V(\Gamma)|$ is to bound the parameter $\mso(\Gamma)$ (or even the order of a vertex-stabiliser $\Aut(\Gamma)_v$) by a constant.
Bounding the order of $\Aut(\Gamma)_v$ is a classical topic in algebraic graph theory, going back to the work of Tutte \cite{tutte} where he proved that $|G_v| \le 48$ for every
connected $G$-arc-transitive graph $\Gamma$.
 This result does not generalised to higher non-prime valencies.
However, a long-standing conjecture of Richard Weiss \cite{WeissConj}
 states that for every fixed valence $d$ there exists
a constant $c_d$ such that every connected $G$-arc-transitive $G$-locally-primitive graph
of valence $d$ satisfies $|G_v| \le c_d$, and thus $\meo(G) \le c_d n$. This observation puts
Theorem~\ref{the:WieThomp} into the context of the Weiss Conjecture (and its recent generalisation
 \cite{PSV} to locally-semiprimitive graphs).
 
 On the other hand, 
it is well known that connected vertex-transitive graphs
can have an arbitrarily large vertex-stabiliser 
while still having the order of an element in a vertex-stabiliser 
bounded by a constant (consider, for
example, the family of lexicographic products $C_n[2K_1]$ of a cycle $C_n$
with the edgeless graph on two vertices, where the order of the vertex-stabiliser
grows exponentially with $n$ but the automorphisms fixing a vertex have order at most $4$). 
However, there are no known infinite families of
connected vertex-transitive graphs $\{\Gamma_n : n\in \NN\}$
 of fixed valence $d$  such that $\mso(\Gamma_n) \to \infty$ as $n\to \infty$.
 This prompted Pablo Spiga~\cite{PabloOral} to ask whether for every valence $d$
 there exists a constant $c_d$ such that every connected $d$-valent vertex-transitive graph 
 $\Gamma$ satisfies $\mso(\Gamma) \le c_d$. Observe that a positive answer to this question
 would also resolve Question~\ref{q:meo}.
Here we consider this question in the case of valence $3$ and prove
 the following result (see Section~\ref{sec:Gv} for the proof).

\begin{theorem}
\label{prop:djoko}
If $\Gamma$ is a connected $3$-valent vertex-transitive graph, then $\mso(G_v) \le 6$.
\end{theorem}

Finally, let us discuss the relationship between the parameters $\ell(\Gamma)$ and 
$\meo(\Gamma)$.
Note that $\ell(g)$ divides $o(g)$ for every permutation $g$.
When the equality $\ell(g) = o(g)$ holds, we call an orbit of length $\ell(g)$ a 
{\em regular orbit} of $\la g \ra$ (or, to simplify terminology, a regular orbit of $g$).
The question which permutation groups have the property that all of their elements
possess a regular orbit  has received considerable attention in group theory (see, for instance, \cite{regularcycles1,regularcycles2,regularcycles3}), especially in the context of primitive permutation groups. This work culminated in \cite{GuestSpiga},
where it was proved that apart of some exceptional cases, every finite primitive permutation group has this property. 
Since a transitive permutation group $G\le \Sym(\Omega)$ is primitive if and only if 
every graph $\Gamma$ with vertex-set $\Omega$ and $G\le\Aut(\Gamma)$
is connected, it is natural to ask to what extent do results about primitive permutation
group  extend to automorphism groups of connected vertex-transitive graphs. In this paper, we consider this question in the setting of cubic graphs.

\begin{theorem}
\label{the:regularorbits}
If $\Gamma$ is a cubic vertex-transitive graph not isomorphic to $K_{3,3}$, then every automorphism of $\Gamma$ has a regular orbit and thus $o(\Gamma) = \ell(\Gamma)$.
If in addition $\Gamma$ is not isomorphic to $K_4$, the cube graph $Q_3$, the Petersen graph, the Pappus graph or the Heawood graph, then for every  $g \in \Aut(\Gamma)$, either $\langle g \rangle$ is transitive on $\V(\Gamma)$ or every regular orbit of $g$ is adjacent to another regular orbit of $g$.
\end{theorem}

Finally, we prove that excluding $K_{3,3}$ and an exceptional family of Split Praeger-Xu graphs, 
regular orbits of an automorphism cover a large part of a cubic vertex-transitive graph.

\begin{theorem}
\label{the:regorbitsratio}
Let $\Gamma$ be a cubic vertex-transitive graph of order $n$ and let $g \in \Aut(\Gamma)$. If $\Gamma$ is not isomorphic to $K_{3,3}$ or a Split Praeger-Xu graph (defined in Section \ref{sec:bound}), then at least $\frac{5}{12}n$ vertices of $\Gamma$ lie on a regular orbit of $g$.
\end{theorem}

 After proving some auxiliary results in Section~\ref{sec:general}, we then prove
Theorem~\ref{prop:djoko} in Section~\ref{sec:Gv},
 Theorem~\ref{the:regularorbits} and Theorem~\ref{the:34} for the case of cubic graphs in Section~\ref{sec:cubic}, 
 Theorem~\ref{the:mainorbits} and  Theorem~\ref{the:regorbitsratio} in Section~\ref{sec:bound},
 and
 Theorem~\ref{the:34} for the case of quartic graphs in Section~\ref{sec:quartic}.


%


\section{Auxiliary results}
\label{sec:general}

In this section we prove a few easy auxiliary results about the parameter introduced in Section~\ref{sec:intro}. We begin with an observation about a relation between these parameters. 
 
\begin{lemma}
\label{lem:basic1}
Let $G$ be a transitive permutation group on a set $\Omega$, let $g\in G$, let
$s(g)$ be the length of a shortest orbit of the cyclic group $\langle g \rangle$
and let $s(G) = \max\{ s(g) : g\in G \}$. Then:
\begin{align}
\label{x1}
  \ell(g) & \le o(g) \le  s(g) \meo(G_\omega)  \le  \ell(g)\meo(G_\omega); \\
\label{x2}
 \ell(G) & \le  \meo(G) \le s(G)\meo(G_\omega) \le \ell(G) \meo(G_\omega),
\end{align}
where $G_\omega$ is the stabiliser of an arbitrary element $\omega\in \Omega$.
\end{lemma}
\begin{proof}
Let $\omega \in \Omega$, let $C=\langle g \rangle$, let $t = |\omega^C|$ and
 let $m=\meo(G_\omega)$; since $G$ is transitive on $\Omega$, the parameter $m$ does not depend on the choice of $\omega$.
Since $s(g) \le \ell(g) \le o(g)$ holds, in order to prove (\ref{x1}), it suffices to show that
$o(g) \le s(g) m$. To do this, 
observe that $C_\omega = \langle g^t \rangle$, implying that
$|C_\omega| = o(g^t) \le \meo(G_\omega) = m$.
By the orbit-stabiliser lemma applied to the action of $C$ on $\Omega$ be may thus conclude that
$$
o(g) = |C| = |C_\omega|\, |\omega^C|  \le m |\omega^C|.
$$ 
If  $\omega$ was chosen to be in a shortest orbit of $C$, then $|\omega^C| = s(g)$ and thus $o(g) \le m s(g)$, as required. Part (\ref{x2}) now follows easily from (\ref{x1}) if we  maximise the expressions over all $g\in G$, starting with the rightmost expression 
in the chain of inequalities of (\ref{x1}) and then proceeding towards the left-hand side.
\end{proof}

The following lemma and its corollary give us an elementary but useful tool that can be used to bound the order of a permutation by a linear function of the degree of the permutation group. Corollary~\ref{cor:TW} is then used in Section~\ref{sec:TW} to prove Theorem~\ref{the:WieThomp}

\begin{lemma}
\label{lem:mingcd}
If $g$ is a permutation on a set $\Omega$ and $C=\langle g \rangle$, then 
\begin{equation*}
 \frac{o(g)}{\ell(g)} = \frac{ \min \{ |C_\omega| : \omega \in \Omega\} }{ \gcd \{ |C_\omega| : \omega \in \Omega\} }.
\end{equation*}
\end{lemma}

\begin{proof}
Recall that $o(g) = \lcm \{|\omega^C| : \omega \in \Omega\}$ and that
$|\omega^C| = |C| / |C_\omega|$, implying that
$$
o(g) = \lcm \{\frac{|C|} { |C_\omega|} : \omega \in \Omega\} = \frac{|C|}{\gcd \{ |C_\omega| : \omega \in \Omega\}}.
$$
Similarly, $\ell(g) =  \max \{|\omega^C| : \omega \in \Omega\}$ and thus
$$
\ell(g) =  \frac{|C|}{\min \{ |C_\omega| : \omega \in \Omega\}}, 
$$
as claimed.
\end{proof}

Recall that the exponent $\exp(G)$ of a group $G$ is the minimum positive integer $e$ such that $g^e =1$ for every $g\in G$. Clearly $\exp(G)$ equals the least common multiple of
the orders of cyclic subgroup of $G$.

\begin{corollary}
\label{cor:TW}
Let $G$ be a transitive permutation group on a set $\Omega$, let $\omega\in \Omega$, let $p$ be a prime and let $k$ be an integer coprime to $p$ such that $\exp(G_\omega) = k p^\alpha$ for some $\alpha \ge 1$. Then 
$$
 o(g) \le k \ell(g)
$$
for every $g\in G$.
In particular, if $G_\omega$ is a $p$-group, then every element of $G$ has a regular orbit.
\end{corollary}

\begin{proof}
Let $\omega\in \Omega$.
Since $C_\omega$ is a cyclic subgroup of $G_\omega$, we see that
$|C_\omega| $ divides $\exp(G_\omega)$ and thus $|C_\omega|  = t p^\beta$ for some $t\le k$,  $\gcd(t,p) =1$, and $\beta\le \alpha$.
 Now let $\{\omega_1, \ldots, \omega_m\}$ be a complete set of representatives of 
the orbits of $C$ on $\Omega$ indexed in such a way that
$$
|C_{\omega_1}|\, \le \, |C_{\omega_2}|\, \le \, \ldots \, \le \, |C_{\omega_m}|.
$$
Write 
 $|C_{\omega_i}| = t_i p^{\alpha_i}$ with $\gcd(t_i,p) = 1$. As observed above, it follows 
 that $t_i\le k$. Let $j\in\{1,\ldots,m\}$ be such that
$\alpha_j = \min\{ \alpha_i : i\in\{1,\ldots,m\}\}$.
Then
$$
 \gcd \{ |C_\omega| : \omega \in \Omega\} =  \gcd \{ t_ip^{\alpha_i} : i\in \{1,\ldots, m\}\} \ge p^{\alpha_j}.
 $$
On the other hand, $\min \{ |C_\omega| : \omega \in \Omega\} = t_1 p^{\alpha_1} \le t_j p^{\alpha_j} \le k p^{\alpha_j}$. The result then follows by Lemma~\ref{lem:mingcd}.
\end{proof}

\begin{lemma}
\label{lem:mu}
If $G$ is a transitive permutation group on a set $\Omega$ of cardinality $n$ and $g\in G$, then
\begin{equation}
\label{eq:mu}
\mu(g) \le \left( \frac{n}{\ell(g)} -1\right) \mso(G) + 1 \quad \hbox{ and } \quad
\mu(G) \le \left( \frac{n}{\ell(G)} -1\right) \mso(G) + 1.
\end{equation}
\end{lemma}

\begin{proof}
As in Lemma~\ref{lem:basic1}, let $s(g)$ denote the shortest orbits of the group $\langle g \rangle$
and let $s(g)=n_1 \le n_2 \le \ldots \le n_m=\ell(g)$ be the lengths of the orbits of $\langle g \rangle$, counted with multiplicity. Then $\mu(g) = m$ and
$$s(g) + n_2 +\ldots +n_{m-1} + \ell(g) = n,$$
 implying that
$(m-1) s(g) \le s(g) + n_2 +\ldots +n_{m-1} \le n-\ell(g)$.
Now recall that $s(g) \ge \ell(g)/\meo(G_\omega)$ and thus
$$
 m \le \frac{n-\ell(g)}{s(g)} + 1 \le \frac{n-\ell(g)}{\ell(g)}\meo(G_\omega) + 1,
$$
proving the first inequality in (\ref{eq:mu}). The second inequality now follows
if we choose $g$ to be an element of $G$ with the smallest number of orbits.
\end{proof}

\section{Vertex-stabilisers of cubic vertex-transitive graphs and proof of Theorem~\ref{prop:djoko}}
\label{sec:Gv}

In this section we prove Theorem \ref{prop:djoko}, which states that 
$
\mso(\Gamma) \le 6
$
for every connected $3$-valent vertex-transitive graph.
The proof is based on the work of Djokovi\'c on
$4$-valent arc-transitive graphs \cite{djoko}, and the splitting and merging operation, introduced in \cite{CVT}, which links the cubic vertex- but not arc-transitive graphs with a class of $4$-valent arc-transitive graphs. 

\begin{figure}[h!]
\centering
\includegraphics[width=0.8\textwidth]{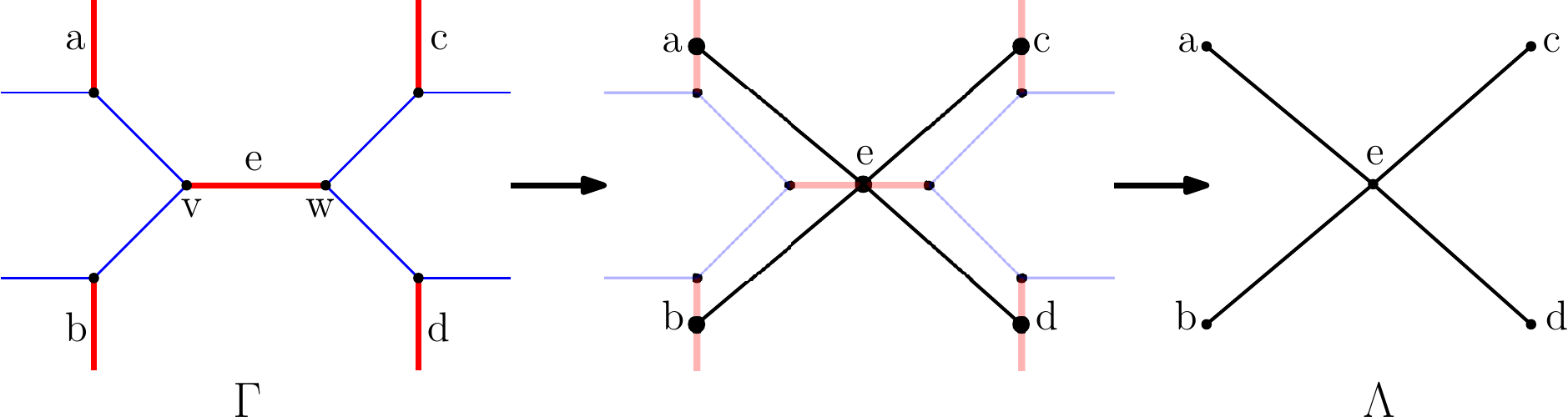}
\label{fig:bounds}
\caption{The neighbourhood of an edge with endpoints $v$ and $w$ in $\Gamma$, and the neighbourhood of the corresponding vertex in $\Lambda$.}
\end{figure}



For the rest of this section, let $\Gamma$ be a cubic vertex-transitive graph, let $G=\Aut(\Gamma)$ and let $v$ be a vertex of $\Gamma$. We need to prove that
$o(g) \le 6$ for every $g\in G_v$.

First, observe that if $\Gamma$ is isomorphic to a prism or a M\"{o}bius ladder, then the result follows, as a vertex stabilizer is isomorphic to either $S_3$, if $\Gamma$ is arc-transitive, or $\ZZ_2$ otherwise. We may thus assume without loss of generality that $\Gamma$ is not isomorphic to a prism or a M\"{o}bius ladder. 

Suppose that $G$ has $m$ orbits in its action on the arcs of $\Gamma$. Then, since $\Gamma$ is vertex-transitive, we have $m \in \{1,2,3\}$.

If $m=3$, then the connectivity of $\Gamma$ implies that the stabilizer $G_v$ in $G$ of any vertex $v \in \V(\Gamma)$ is trivial, and the result follows.

If $m=1$, then $\Gamma$ is arc-transitive. By a result of Djokovi\'c and Miller \cite{DjM} (which is based on the celebrated work of Tutte \cite{tutte} on cubic arc-transitive graphs), $G_v$ is isomorphic to either $\ZZ_3$, $S_3$, $S_3 \times S_2$, $S_4$ or $S_4 \times S_2$. In none of the five possible cases, an element of $G_v$ has order greater than $6$.

If $m=2$, then $G_v$ must fix an edge $x$ incident to $v$. Let $\mathcal{T}$ be the orbit of $x$ under the action of $G$ and observe that $\mathcal{T}$ is a perfect matching in $\Gamma$. Moreover, $G$ is transitive on both $\A(\Gamma) \setminus \mathcal{T}^*$ and $\mathcal{T}^*$, where $\mathcal{T}^*$ denotes the arcs of $\Gamma$ underlying an edge in $\mathcal{T}$.

 We can thus construct a connected tetravalent graph $\Lambda$ as follows. Let $\mathcal{T}$ be the vertex-set of $\Lambda$ and let $uv \in \mathcal{T}$ be adjacent to $ab \in \mathcal{T}$ in $\Lambda$ if and only if there exist an edge $xy \in \E(\Gamma)$ such that $x \in \{u,v\}$ and $y \in \{a,b\}$. Informally, $\Lambda$ is constructed from $\Gamma$ by contracting every edge in $\mathcal{T}$ (see Construction 7 in Section 4 of \cite{CVT} for details about this construction and the graph $\Lambda$).
%
Since $\Gamma$ is neither a prism or a M\"{o}bius ladder, then by \cite[Lemma 9]{CVT} $\Lambda$ is a simple graph. 
Clearly every $g\in G$ induces an automorphism of $\Lambda$. Since $G$ is transitive on $\A(\Gamma)\setminus \mathcal{T}^*$, we see that $G$ acts arc-transitively on $\Lambda$. 

Let $e \in \mathcal{T}$ have endpoints $v$ and $w$ and assume the notation in Figure \ref{fig:bounds}. Then $N:=\{a, b, c, d\}$ is the set of neighbours of $e$ in $\Lambda$. 
Since $G$ acts transitively on the arcs of $\Lambda$, it follows that the permutation group
$G_e^{N}$ induced by the action of $G_e$ on the set $N$ is transitive. Observe also that
$\{\{a,b\},\{c,d\}\}$ is a $G_e$-invariant partition of $N$ and that $G_v$ is isomorphic to the subgroup of $G_{e}$ that fixes $\{a,b\}$ and $\{c,d\}$ set-wise. Hence, $G_e^N$ is permutation isomorphic to
one of the three imprimitive transitive groups of degree four:  $\D_4$, $\ZZ_4$ or $\ZZ_2^2$.

If  $G_{e}^{N}\cong \ZZ_4$ or $\ZZ_2^2$, then it follows from the connectivity of $\Lambda$ that $G_{e} \cong \ZZ_4$ or $G_{e} \cong \ZZ_2^2$, respectively.   
Suppose now that  $G_{e}^N \cong \D_4$. Then, $(G_{e}, G_{(e,a)}, G_{ea})$ (where $(e,a)$ denote the arc in $\Lambda$ and $ea$ denotes the underlying undirected edge) is a dihedral amalgam of type $(4,2)$ (see \cite{djoko}). If follows from the main theorem of \cite{djoko} that the subgroup $H$ of $G_{e}$ that fixes both $\{a,b\}$ and $\{c,d\}$ set-wise satisfies the following:
\begin{eqnarray*}
H &=& \langle a_0, a_1, a_2, \ldots, a_{n-1} \rangle ,\\
a_i^2 &=& 1, \qquad 0 \leq i \leq n-1, \\
\text{[}a_i,a_j] &=& r_{i,j} \qquad 0 \leq i,j \leq n-1,
\end{eqnarray*}
where each $r_{i,j}$ is in the center $\Z(H)$ of $H$ and has order at most $2$. In particular $N/\Z(H)$ is an elementary abelian $2$-group implying that $g^2 \in \Z(H)$ for every $g \in H$. However since $N$ is generated by involutions, the exponent of $\Z(H)$ is at most $2$, implying that the exponent of $N$ is at most $4$. Then, since $G_v \cong H$, we see that $G_v$ has exponent at most $4$.

We have shown that in all possible cases, the order of an element $g \in G_v$ has order at most $6$, thus concluding the proof of Theorem~\ref{prop:djoko}.


\section{Regular orbits of cubic vertex-transitive graphs and proof of Theorem~\ref{the:regularorbits}}
\label{sec:cubic}

 For a graph $\Gamma$ and a group $G \leq \Aut(\Gamma)$, we let $\Gamma/G$ be the graph whose vertices are the $G$-orbits of vertices in $\Gamma$ and two such $G$-orbits $X$ and $Y$ are adjacent if there is a vertex in $X$ adjacent to a vertex in $Y$. We will show that with a few exceptions, every automorphism of a cubic vertex-transitive graph admits a regular orbit (Theorem \ref{the:1reg}), and if the group generated by such an automorphism is non-transitive, then every regular orbit is adjacent to another regular orbit. These two results combined give us Theorem \ref{the:regularorbits}. We would like to point out that Theorem \ref{the:1reg} follows at once from Corollary \ref{cor:TW} if we restrict it to cubic graphs that are vertex- but not arc-transitive (as in this case the stabiliser of any vertex is a $2$-group). However, the arc-transitive case requires a little bit more work. We will need to prove a series of rather simple, but useful lemmas, about the relative sizes of adjacent orbits. 
 
 In what follows, we let $K_4$ denote the complete graph on $4$ vertices, $K_{3,3}$ the complete bipartite graph with $3$ vertices in each part, and $Q_3$ denote the tridimensional cube graph. For integers $n$ and $k$, the generalised Petersen graph ${\rm{GP}}(n,k)$ is the graph with vertex set $\{x_i,y_i \mid i \in \ZZ_n\}$ and with edges of the form $x_ix_{i+1}$, $x_iy_i$ and $y_iy_{i+k}$. The well-known Pappus graph and Heawood graph are depicted in Figure \ref{fig:examples}. The following three lemmas are useful special cases of \cite[Theorem 5]{signature}.


\begin{lemma}
\label{lem:smallgirthAT}
If $\Gamma$ is a cubic arc-transitive graph of girth smaller than $6$, then $\Gamma$ is isomorphic to one of the following: $K_4$, $K_{3,3}$, $Q_3$, the Petersen Graph or the dodecahedron ${\rm{GP}}(10,2)$.
\end{lemma}

\begin{lemma}
\label{lem:4cycles}
Let $\Gamma$ be a cubic vertex-transitive graph of girth $4$. If there is a vertex $u$ such that all three edges incident to $u$ lie on two distinct $4$-cycles, then $\Gamma$ is isomorphic to $K_{3,3}$ or $Q_3$.
\end{lemma}

\begin{lemma}
\label{lem:3cycles}
Let $\Gamma$ be a cubic vertex-transitive graph of girth $3$. If any of the following conditions hold, then $\Gamma \cong K_4$:
\begin{enumerate}
\item there is a vertex $u$ such that all three edges incident to $u$ lie on a $3$-cycle;
\item there is an edge $e$ that lies on two distinct $3$-cycles.
\end{enumerate}
\end{lemma}

\begin{lemma}
\label{lem:no23}
Let $\Gamma$ be a cubic vertex-transitive graph, let $g \in \Aut(\Gamma)$ and let $G = \langle g \rangle$. Let $u$ and $v$ be two adjacent vertices belonging to distinct $G$-orbits. If $i|u^G|=3|v^G|$ for some $i \in \{1,2\}$, then $\Gamma$ is arc-transitive. In particular, if $i = 2$ then $\Gamma \cong K_{3,3}$.
\end{lemma}

\begin{proof}
Let $n \in \ZZ$ be such that $3n = |u^G|$ and $i\cdot n = |v^G|$. Observe that $v = v^{g^{2n}} = v^{g^{4n}}$ but that $u$, $u^{g^{2n}}$ and $u^{g^{4n}}$ are all distinct. Moreover, since $u \sim v$, we have $u^{g^{2n}} \sim v^{g^{2n}}$ and $u^{g^{4n}} \sim v^{g^{4n}}$. Then $v$ is adjacent to $u$, $u^{g^{2n}}$ and $u^{g^{4n}}$ and since $u^{g^{4n}}=u^{g^{3n}g^n}=u^{g^{n}}$, we see that $\Gamma(v) = \{u,u^{g^{n}},u^{g^{2n}}\}$. The group $\langle g^{2n} \rangle$ then fixes $v$ while permuting its neighbours. It follows that all three edges incident to $v$ belong to the same $\langle g^{2n} \rangle$-orbit. Then, since $\Gamma$ is vertex-transitive, it must also be edge-transitive. Moreover, a cubic graph that is both vertex- and edge-transitive must necessarily be arc-transitive.

Now, suppose $i=2$. Then $|v^G| = 2n$ and thus $v \neq v^{g^n}$. We have shown that $\Gamma(v) = \{u,u^{g^{n}},u^{g^{2n}}\}$, and thus $\Gamma(v^{g^n}) = \{u^{g^{n}},u^{g^{2n}},u^{g^{3n}}\}=\{u,u^{g^{n}},u^{g^{2n}}\}$. Then every vertex in $\{v,v^{g^n}\}$ is adjacent to every vertex in $\{u,u^{g^{n}},u^{g^{2n}}\}$. That is, $\Gamma$ contains a copy of $K_{2,3}$ and by Lemma \ref{lem:smallgirthAT}, $\Gamma$ is isomorphic to $K_4$, $K_{3,3}$ or $Q$. However, neither $K_4$ nor $Q$ contain a subgraph isomorphic to $K_{2,3}$. We conclude that $\Gamma \cong K_{3,3}$.
\end{proof}

\begin{lemma}
\label{lem:orbitratio}
Let $\Gamma$ be a cubic vertex-transitive graph other than $K_{3,3,}$. Let $G \leq \Aut(\Gamma)$ be cyclic, and let $u$ and $v$ be two adjacent vertices belonging to distinct $G$-orbits. If $|u^G| \geq |v^G|$, then $|u^G| = i|v^G|$ for some $i \in \{1,2,3\}$.
 Moreover, if $i \neq 1$, then every vertex in $u^G$ has exactly one neighbour in $v^G$, while a vertex in $v^G$ has $i$ neighbours in $u^G$.
\end{lemma}

\begin{proof}
Let $n = |u^G|$ and $m = |v^G|$ and note that $u^{g^i} = u$ if and only if $n \vert i$. Similarly, $v^{g^i}=v$ if and only if $m \vert i$. It follows that $v$ is adjacent to $u^{g^{i m}}$ for all $i \in \ZZ$. Let $\lambda$ be the number of neighbours of $v$ in $u^G$. Clearly, $\lambda \in \{1,2,3\}$. 
If $\lambda = 1$, then $u^{g^m} =u$ and thus $n | m$. Since by hypothesis $n \geq m$, we see that $n=m$.
If $\lambda = 2$, then $u^{g^{2m}} =u$. It follows that $n | 2m$ and thus $n = 2m$ or $n = m$.
Finally, if $\lambda = 3$, then $u^{g^{3m}} = u$ and $n \vert 3m$. Then $n = 3m$, $n=\frac{3}{2}m$ or $n = m$. However, if $n=\frac{3}{2}m$, then by Lemma \ref{lem:no23}, $\Gamma \cong K_{3,3}$, a contradiction. It follows that $n = 3m$ or $n = m$.
Therefore $|u^G| = i|v^G|$ for some $i \in \{1,2,3\}$. 

To prove the second part of the statement, suppose $i\neq 1$ and that $u$ is adjacent to a vertex $v^{g^j}$ for some $0< j < m$. Note that $v^{g^j}$ is also adjacent to $u^{g^m}$ , $u^{g^j}$ and $u^{g^j+m}$. Since $0< j < m$ and $|u^G|\geq 2m$, we see that the vertices in $\{u, u^{g^m}, u^{g^j},u^{g^j+m}\}$ are all distinct. Thus, $v$ has four neighbours, contradicting $\Gamma$ being a cubic graph. It follows that a vertex in $u^G$ has exactly one neighbour in $v^G$. In particular, this means there are exactly $|u^G|=im$ edges between $u^G$ and $v^G$. and thus a vertex in $v^G$ has exactly $i$ neighbours in $u^G$.
\end{proof}

\begin{corollary}
\label{cor:13}
Let $\Gamma$ be a cubic vertex-transitive graph, let $G= \langle g \rangle \in \Aut(\Gamma)$ and let $u$ and $v$ be two adjacent vertices. If $|u^G|=3|v^G|$ then $u^G$ is the only neighbour of $v^G$ in $\Gamma/G$.
\end{corollary}

\begin{proof}
By Lemma~\ref{lem:orbitratio}, every vertex in $v^G$ has three neighbours in $u^G$, and thus no neighbours in orbits other than $u^G$. Hence $u^G$ is the only neighbours of $v^G$ in $\Gamma/G$.
\end{proof}

In the next theorem, we prove the first part of Theorem~\ref{the:regularorbits}.

\begin{theorem}
\label{the:1reg}
If $\Gamma$ is a cubic vertex-transitive graph other than $K_{3,3,}$, then every $g \in \Aut(\Gamma)$ admits a regular orbit and every orbit of $\langle g \rangle$
has size $\frac{\ell(g)}{k}$ for some $k \in \{1,2,3,4,6\}$.
\end{theorem}

\begin{proof}
Let $g \in \Aut(\Gamma)$, $G = \langle g \rangle$ and let $u \in \V(\Gamma)$ be such that $\ell(g)=|u^G| \geq |v^G|$ for all $v \in \V(\Gamma)$. We will show that the size of any $G$-orbit divides $|u^G|$. Let $v \in \V(\Gamma) \setminus \{u^G\}$. 

Consider the quotient graph $\Gamma/G$. Since $\Gamma$ is connected so is $\Gamma/G$ and thus there exists a $u^Gv^G$-path 
$W = u_0^G\,u_1^G\, \ldots\, u_n^G$ where $u_0 = u$, $u_n = v$ and each $u_i$ is adjacent to $u_{i+1}$ in $\Gamma$.

Suppose that for some $i,j \in \{0,\ldots,n\}$ and $m \in \ZZ$, we have $|i-j|=1$ and $|u_i^G| =  3|u_j^G| = 3m$. Then by Corollary \ref{cor:13}, $u_i^G$ is the only neighbour of $u_j^G$ in $\Gamma/G$. It follows that $j = 0$ (and thus $i=1$) or $j = n$ (and thus $i=n-1$). However, since $|u_0^G| \ge |u_1^G|$ (by our assumption on $u$), we see that $j\not = 0$ and thus $j=n$. 
This together with Lemma \ref{lem:orbitratio} implies that for all $i \in \{0,\ldots,n-2\}$ we have $|u_i| = k |u_{i+1}|$ for some $k \in \{\frac{1}{2},1,2\}$. Then $|u_{n-1}^G| = \frac{1}{2^r} |u_0^G|$ for some integer $r$, $0\leq r \leq n-1$, and
thus 
$|v^G|= |u_n^G|=\frac{1}{3\cdot 2^r}|u_0^G| = \frac{1}{3\cdot 2^r}|u^G|.$


Now, if no two orbits $u_i^G$ and $u_j^G$ satisfy $|u_i^G| = 3|u_j^G|$, then by Lemma \ref{lem:orbitratio}, we have $|u_i| = k_i |u_{i+1}|$ with $k_i \in \{\frac{1}{2},1,2\}$
 for all $i \in \{0,\ldots,n-1\}$, and thus  $|u^G|=\frac{1}{2^r}|v^G|$ for some $r \in \{0,\ldots,n\}$. 
 
Let $m=|v^G|$. We have shown above
 that $o(g) = \ell(g) = |u^G| = km$ where $k$ is an integer
not divisible by any prime larger than $3$. In particular, $u^G$ is a regular orbit of $g$.
Observe that $g^m \in G_v$ and that $|u^{g^m}| =k$.
In particular, $g^m$ is an element fixing a vertex of order at least $k$. But then 
Theorem~\ref{prop:djoko}, which was proved in Section~\ref{sec:Gv}, implies that $k\le 6$
and thus $k\in \{1,2,3,4,6\}$.
\end{proof}

\begin{corollary}
\label{cor:d=3}
Theorem~\ref{the:34} holds in the case $d=3$. That is,
 $\meo(\Gamma) \le n$ for every cubic vertex-transitive graphs $\Gamma$ with $n$ vertices.
\end{corollary}

\begin{proof}
Observe that in the case $\Gamma\not\cong K_{3,3}$  this follows directly from the existence of a regular orbit guaranteed by
Theorem~\ref{the:1reg}.
 On the other hand, if $\Gamma\cong K_{3,3}$, then $\meo(K_{3,3}) =6 = n$.
\end{proof}



\begin{remark}
Observe that for every $k \in \{1,2,3,4,6\}$, there exists a graph $\Gamma$ and an automorphism $g \in \Aut(\Gamma)$ such that $g$ has an orbit of size $\ell(g)/k$. The Pappus graph admits an automorphism $g$ with orbits of sizes $6$, $3$, $2$ and $1$ while the Heawood graph admits an automorphism  with orbits of size $4$, $2$ and $1$ (see Figure \ref{fig:examples}).
\end{remark}

\begin{figure}[h!]
\centering
\includegraphics[width=0.7\textwidth]{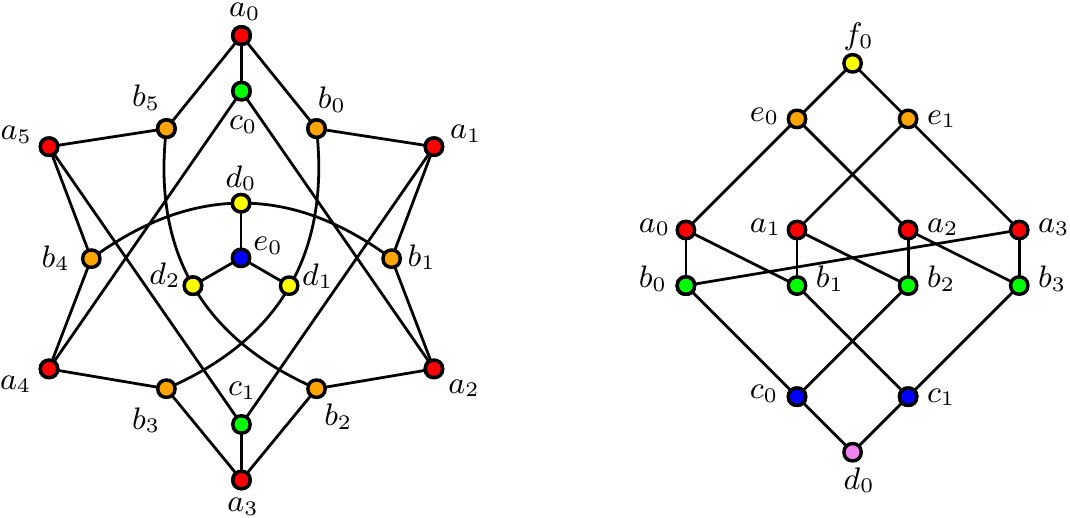}
\caption{The Pappus graph (left) and the Heawood graph (right). For each graph, vertices of the same colour belong to the same orbit under the action of the automorphism $\varphi$ that acts by adding $1$ to the sub-index of each vertex.}
\label{fig:examples}
\end{figure}

\begin{lemma}
\label{lem:squares}
Let $\Gamma$ be a cubic vertex-transitive graph, let $G \leq \Aut(\Gamma)$ be a cyclic group, and let $u,v,w \in \V(\Gamma)$ such that $u \sim v$ and $u \sim w$. If $|u^G|= i|v^G| = i|w^G|$ for some $i \in \{2,3\}$ then the girth of $\Gamma$ is at most $4$ and if $i = 3$, then $\Gamma$ is isomorphic to $K_4$, $K_{3,3}$ or the cube graph $Q$. 
\end{lemma}

\begin{proof}
Let $g$ be a generator of $G$ and set $ m = |v^G| =|w^G|$. Then $u^{g^m} \neq u$, but $v^{g^{m}} = v$ and $w^{g^{m}} = w$. It follows that $u^{g^m} \sim v$ and $u^{g^m} \sim w$ and thus $(u,v,u^{g^m},w)$ is a $4$-cycle. Moreover, if $i=3$, then by Lemma \ref{lem:no23}, $\Gamma$ is arc-transitive and by Lemma \ref{lem:smallgirthAT}, $\Gamma$ is isomorphic to $K_4$, $K_{3,3}$ or $Q$.
\end{proof}

The following remark, which will be used in the proof of Theorem~\ref{the:mainorbits}, is a consequence of Corollary \ref{cor:13}, Lemma \ref{lem:orbitratio}, and Theorem \ref{the:1reg}. 

\begin{lemma}
\label{lem:pairs}
Suppose $\Gamma$ is a cubic vertex-transitive graph other than $K_4$, $K_{3,3}$ or the cube $Q_3$, let $g \in \Aut(\Gamma)$. Then the following hold.
\begin{enumerate}
\item for $i \in \{1,2\}$, an orbit of size $\ell(g)/3i$ is adjacent to only one orbit, which has size $\ell(g)/i$;
\item for $i \in \{1,2\}$ an orbit of size $\ell(g)/i$ is adjacent to at most one orbit of size $\ell(g)/3i$;
\item if $v^G$ has size $\ell(g)/4$ then there exists an orbit $w^G$ of size $\ell(g)/2$ (not necessarily adjacent to $v^G$).
\end{enumerate}
\end{lemma}

\begin{proof}
Let $u \in \V(\Gamma)$ be such that $|u^G| = \ell(g)$ and let $v^G$ be any other orbit. Since $\Gamma$ is connected there exists a $v^Gu^G$-path $W = v_0^G\,v_1^G\, \ldots\, v_n^G$ in $\Gamma/G$ where $v_0 = v$, $v_n = u$ and each $v_i^G$ is adjacent to $v_{i+1}^G$. 

To show that (1) holds, suppose $|v_0^G| = \ell(g)/3i$. By Lemma \ref{lem:orbitratio}, if $j \in \{0,\ldots,n\}$, then $|v_{j+1}^G| \geq |v_j^G|$ implies $|v_j^G| = k|v_j^G|$ for some $k \in \{1,2,3\}$. Then, since $|v_n^G| =  3i|v_0^G|$, there must exist $j \in \{0,\ldots,n\}$ such that $|v_{j+1}^G| = 3|v_j^G|$. By Corollary \ref{cor:13}, $v_{j+1}^G$ is the only orbit adjacent to $v_j^G$, which implies that $j =0$. That is,  $v_j^G=v_0^G=v^G$. Thus, $v^G$ is adjacent to only one orbit of size $3|v^G|=3\cdot\ell(g)/3i=\ell(g)/i$. Therefore (1) holds.

Now, to show that (3) holds, suppose $|v_0^G| = \ell(g)/4$. Let $j \in \{1,\ldots,n\}$ be the smallest integer for which $|v_j^G| \neq \ell(g)/4$. If $|v_j^G| < |v_{j-1}^G| = \ell(g)/4$, then by Lemma \ref{lem:orbitratio}, $|v_j^G| = |v_{j-1}^G|/k = \ell(g)/4k$ holds for some $k \in \{2,3\}$, but this contradicts Theorem \ref{the:1reg}. Then $|v_j^G| > |v_{j-1}^G| = \ell(g)/4$, and by Lemma~\ref{lem:orbitratio} we have $|v_j^G| = k\cdot\ell(g)/4$ for some $k \in \{2,3\}$. However, by Theorem \ref{the:1reg}, $k \neq 3$. It follows that $|v_j^G| = 2\ell(g)/4=\ell(g)/2$ and thus (3) holds.

Finally, let us show that (2) holds. Suppose $v^G=\ell(g)/i$ for some $i \in \{1,2\}$. If $v$ has two neighbours $w_1$ and $w_2$ such that $|v^G|=3|w_1^G|=3|w_2^G|$, then by Lemma \ref{lem:squares}, $\Gamma$ is isomorphic to $K_4$, $K_{3,3}$ or $Q_3$.
\end{proof}


\subsection{Proof of Theorem~\ref{the:regularorbits}}
We are now ready to finish the proof of Theorem~\ref{the:regularorbits}. 
For the rest of the section, let $\Gamma$ be a cubic vertex-transitive graph not isomorphic to $K_{3,3}$, let $g\in \Aut(\Gamma)$ and let $G=\langle g \rangle$. We may assume that
$G$ is not transitive on the vertices of $\Gamma$.
In Theorem~\ref{the:1reg} we have already proved that $g$ has at least one regular orbit,
say $u^G$.
Let us now assume in addition that
 $\Gamma$ is not isomorphic to $K_4$, the cube graph $Q_3$,  the Petersen graph, the Pappus graph or the Heawood graph.
 We then need to show that  every regular orbit of $g$ other than $u^G$ is adjacent to  $u^g$ in $\Gamma/G$.

 Since the order of $G$ is smaller than $n$, $u^G$ is not the only orbit of $G$ and since $\Gamma$ is connected, $u^G$ must be adjacent to another orbit in $\Gamma/G$. If any of the neighbouring orbits of $u^G$ has size $|u^G|$, then the claim is proved. We shall thus assume that this is not the case.
 Thus, if $v^G$ is an orbit adjacent to $u^G$, Lemma~\ref{lem:orbitratio} implies that $|u^G|=k|v^G|$ for some $k \in \{2,3\}$ and every vertex of $u^G$ is adjacent to precisely $k$ vertices in $v^G$.
There are three cases (which are divided in a total of 7 subcases) to be considered, depending on the numbers of neighbours of $u^G$ in $\Gamma/G$. 

{\em Case 1: $u^G$ is adjacent to exactly one other orbit $v^G$, where $u \sim v$}. 
 Then $|u^G| = k |v^G|$ for some $k \in \{2,3\}$

Case 1.1: $|u^G| = 2 |v^G|$. Since $u$ has only one neighbour in $v^G$ and no neighbours in the $G$-orbits apart from $v^G$ and $u^G$, it follows that it has two neighbours in $u^G$, which must be of the form $u^{g^i}$ and $u^{g^{-i}}$ for some $i \in \{1,\ldots,|u^G|-1|\}$.
Let $m = |v^G|$, and observe that $v \sim u$ and $v \sim u^{g^m}$, and  the edges $vu$ and $vu^{g^m}$ belong to the same $G$-orbit. Since $\Gamma$ is vertex-transitive, there exists $h \in \Aut(\Gamma)$ mapping $v$ to $u$. Clearly, one edge in $\{vu,vu^{g^m}\}$ is mapped by $h$ to an edge in $\{u^{g^i},u^{g^{-i}}\}$. It follows that the edges $uu^{g^i}$, $uu^{g^{-i}}$, $vu$ and $v^{gm}$ belong to the same $\Aut(\Gamma)$-orbit. That is, $\Gamma$ is edge-transitive, and since it is cubic and vertex-transitive, it must be arc-transitive. If the girth of $\Gamma$ is smaller than $6$, then \ref{lem:smallgirthAT} implies that $\Gamma$ is isomorphic to $K_4$, $K_{3,3}$, $Q_3$, the Petersen graph or the dodecahedron graph $\rm{GP}(10,2)$. However, among these five graphs only the dodecahedron graph has the property that every automorphism admits two adjacent regular orbits. Thus $\Gamma$ is isomorphic  to $K_4$, $K_{3,3}$, $Q_3$, the Petersen graph.
Otherwise, observe that for $\alpha \in \{-1,1\}$, $C_{\alpha} = (v,u,u^{g^{\alpha i}},v^{g^{\alpha i}},u^{g^{\alpha i+m}},u^{g^m})$ is a cycle and thus $\Gamma$ has girth $6$. Moreover, $vu$ and $vu^m$ each lie on both $C_0$ and $C_1$ (see Figure \ref{fig:cycles}). Therefore, if we let $w$ be the third neighbour of $v$ (that does not belong to $u^G$), then the edge $vw$ must lie on one $6$-cycle $C \neq C_{\alpha}$. Clearly $C$ must visit either $vu$ or $vu^m$, which implies that one edge incident to $u$ (and therefore every edge of $\Gamma$ by arc-transitivity) lies on at least three $6$-cycles. It then follows from \cite[Lemma 4.2]{ATgirth} that $\Gamma$ is isomorphic to Pappus Graph, the Heawood graph, the M\"{o}bius-Kantor graph $\rm{GP}(8,3)$ or the Desargues graph $\rm{GP}(10,3)$. I can be verified that every automorphism of the graphs $\rm{GP}(8,3)$ or $\rm{GP}(10,3)$ admit two adjacent regular orbits. 

Case 1.2: $|u^G| = 3 |v^G|$. Observe that $v$ has three distinct neighbours in $u^G$, which implies that $u^G$ and $v^G$ are the only $G$-orbits of $\Gamma$. As in the previous case, $u$ is adjacent to $u^{g^i}$ and $u^{g^{-i}}$ for some $i \in \{1,\ldots,|u^G|-1|\}$. Furthermore $v$ is adjacent to $u$, $u^{g^m}$ and $u^{g^{2m}}$ where $m = |v^G|$. If $m = 1$ or $m=2$, then $\Gamma$ is isomorphic to $K_4$ or $Q_3$ respectively. If $m \geq 3$, then the edge $uv$ lies on $4$ distinct $6$-cycles (see Figure \ref{fig:cycles}), and by \cite[Lemma 4.2]{ATgirth}, $\Gamma$ is isomorphic to Heawood graph or the Pappus graph.


{\em Case 2: $u^G$ has exactly two neighbouring orbits $v^G$ and $w^G$, with $u \sim v$ and $u \sim w$}. We have three subcases.

Case 2.1: $|u^G|=3|v^G|=3|w^G|$. By Lemma \ref{lem:squares}, $\Gamma$ is isomorphic to either $K_4$, $K_{3,3}$, or $Q$.

Case 2.2: $|u^G|=2|v^G|=2|w^G|$. By Lemma \ref{lem:orbitratio}, $u$ has exactly one neighbour in each $v^G$ and $w^G$. Since $\Gamma$ is cubic, the third neighbour of $u$ must belong to $u^G$. Indeed, $u \sim u^{g^m}$ where $m = |v^G| = \frac{1}{2}|u^G|$. Then $(u,u^{g^m},v)$ and $(u,u^{g^m},w)$ are $3$-cycles in $\Gamma$, and the edge $uu^{g^m}$ lies on both of them. Then, by Lemma \ref{lem:3cycles}, $\Gamma \cong K_{3,3}$.

Case 2.3: $|u^G|=2|v^G|=3|w^G|$. Observe that $u \sim u^{g^m}$  where $m = |v^G| = \frac{1}{2}|u^G|$. Then $(u,u^{g^m},v)$ is a $3$-cycle, and since $|u^G|=3|w^G|$, $\Gamma$ is arc-transitive by Lemma \ref{lem:no23}. Then by Lemma \ref{lem:smallgirthAT}, $\Gamma \cong K_4$.

Case 3: $u^G$ has three distinct neighbouring orbits, $v^G$, $w^G$, $x^G$ where $u$ is adjacent to $v$, $w$ and $x$. 

Case 3.1: One of these orbits has size $1/3|u^G|$. Then by Lemma \ref{lem:no23}, $\Gamma$ is arc-transitive. Furthermore, by the pigeon hole principle, two of these orbits must have size $1/k|u^G|$ for some $k \in \{2,3\}$. It follows that $\Gamma$ has a $4$-cycle and by Lemma \ref{lem:smallgirthAT}, $\Gamma$ is isomorphic to $K_4$, $K_{3,3}$ or $Q$.

Case 3.2: $|u^G|=2|v^G|=2|w^G|=2|x^G|$. Observe that every edge incident to $u$ lies on two distinct $4$-cycles. If the girth of $\Gamma$ is $4$, then by Lemma \ref{lem:4cycles}, $\Gamma$ is isomorphic to $K_{3,3}$ or the cube graph $Q$. If the girth of $\Gamma$ is $3$, then $u$ must lie on a $3$-cycle. That is, two of the neighbours of $u$ must be adjacent. Without loss of generality, assume $v \sim w$. Then both $(u,v,w)$ and $(u^{g^m},v,w)$ are $3$-cycles. In particular, all three edges incident to $v$ lie on a $3$-cycle. Since $\Gamma$ is vertex-transitive, then every edge incident to $u$ must lie on a $3$-cycle. In particular, this implies that $u$ and $x$ have a common neighbour. This leads us to a contradiction, since neither $v$ nor $w$ is adjacent to $x$ (the neighbourhoods of $v$ and $w$ are $\{u,u^{g^m},w\}$ and $\{u,u^{g^m},v\}$, respectively).

This finishes the proof of Theorem~\ref{the:regularorbits}.

\begin{figure}
\includegraphics[width=0.8\textwidth]{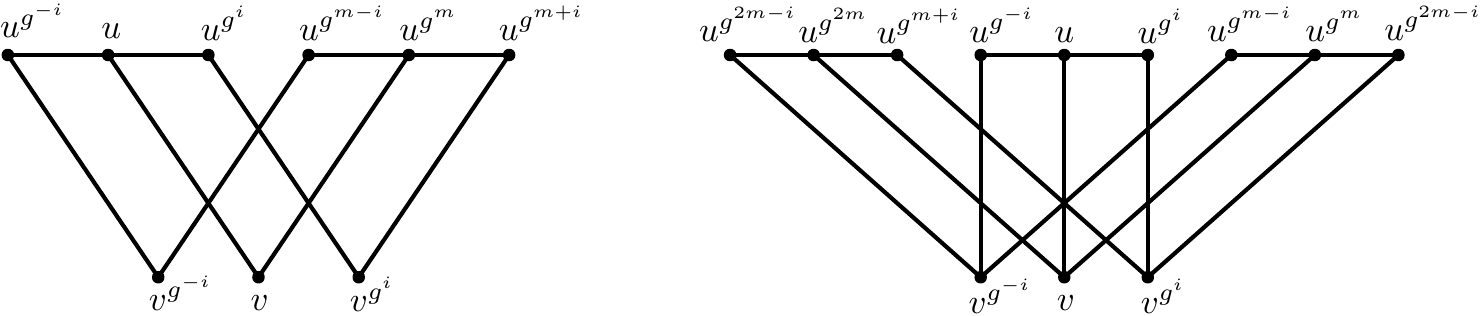}
\caption{Cycles of length $6$ through $u$ according to Cases $1.1$ and $1.2$.}
\label{fig:cycles}
\end{figure}

\section{Bounding the number of orbits and a proof of Theorem~\ref{the:mainorbits}}
\label{sec:bound}

The aim of this section is to prove Theorem~\ref{the:mainorbits}.
That is, we want to show that
\begin{equation}
\label{eq:mug}
\mu(g) \le \frac{17n}{6\, o(g)}
\end{equation}
holds for every automorphism $g$ of a cubic vertex-transitive graph with $n$ vertices.
The second claim then follows by applying this inequality to an automorphism $g$ of largest order. Note that in view of Theorem~\ref{the:1reg}, it
suffices to prove the inequality with
the parameter $o(g)$ substituted by $\ell(g)$ (observe that the inequality (\ref{eq:mug})
clearly holds for $\Gamma = K_{3,3}$
and an automorphism $g$ of $K_{3,3}$ without a regular orbit).

We begin by proving the inequality (\ref{eq:mug}) for the case where $\Gamma$ belongs
to the family of {\em Split Praeger-Xu graphs}, which we now introduce.

For an integer $r\ge 3$, let $\dPX(r,1)$ denote the directed graph with the vertex-set $\ZZ_r \times \ZZ_2$ and with a directed edge pointing from $(x,i)$ to $(x+1,j)$ for every $x\in \ZZ_r$ and $i,j\in \ZZ_2$. For an integer $s$, $2\le s\le r-1$, we let $\dPX(r,s)$ be the directed graph
whose vertex-set is the set of all directed paths of length $s$ in $\dPX(r,1)$ and with
a directed edge pointing from an $s$-path $(u_0, u_1, \ldots, u_s)$ to the successor $s$-paths $( u_1, \ldots, u_s,u_{s+1})$ and $( u_1, \ldots, u_s,v_{s+1})$, where $u_{s+1}$ and $v_{s+1}$ are the two out-neighbours of $u_s$ in $\dPX(r,1)$.
We should point out that the directed graphs $\dPX(r,s)$ were first introduced in \cite{PraDir} in a slightly different way and were denoted $C_2(r,s)$. Several equivalent descriptions of the directed graphs $\dPX(r,s)$ and their undirected counterparts were discussed in \cite{JajPotWilPX}. 

 The {\em Split Praeger-Xu graph} $\SPX(r,s)$  is the graph obtained from $\dPX(r,s)$ by splitting each vertex
$u$ of $\dPX(r,s)$ into two vertices, denoted $u_-$ and $u_+$, and by connecting
each $u_-$ with $u_+$ for every vertex $u$ of $\dPX(r,s)$, and every $v_+$ to $u_-$ for every directed edge $(v,u)$ of $\dPX(r,s)$. The splitting operation was introduced in \cite{PotVerDigraphs}, where the graphs $\SPX(r,s)$ appeared under the name ${\rm Pl}^{s-1}({\rm \vec{W}}(r,2))$.

Observe that the automorphism group of $\dPX(r,1)$ is isomorphic to the semidirect product
$\C_2^r \rtimes \C_r$ with the elementary abelian group $\C_2^r$ being generated by automorphisms $\tau_i$, $i\in \ZZ_r$, interchanging the vertices $(i,0)$ and $(i,1)$ while 
fixing all other vertices. Moreover, $\Aut(\dPX(r,1))$ acts in an obvious way as a vertex-transitive group of automorphisms on $\dPX(r,s)$ for every $s$, $1\le s \le r-1$, as well as on $\SPX(r,s)$. In fact, one can easily see that $\Aut(\dPX(r,s)) \cong \Aut(\dPX(r,1))$ and $|\Aut(\dPX(r,s)): \Aut(\SPX(r,s))| = 2$. In this correspondence, an automorphism fixing a vertex of $\dPX(r,s)$ corresponds to an automorphism of $\dPX(r,1)$ fixing a directed $s$-path of $\dPX(r,1)$ and thus belongs to the group $\C_2^r$. Similarly,  an element $g\in \Aut(\SPX(r,s))$ fixing a vertex $v_+$ of $\SPX(r,s))$ corresponds to an automorphism of $\dPX(r,s))$ which fixes a vertex of $v$. In particular,
the exponent of the vertex-stabiliser in $\Aut(\SPX(r,s))$ is $2$, and so
$
\mso(\SPX(r,s)) = 2.
$
We can now apply the inequality (\ref{eq:mu}) in Lemma~\ref{lem:mu} to conclude that
$$
 \mu(g) \leq 2 (\frac{n}{\ell(g)} - 1) + 1 
  < \frac{17n}{6\, o(g)}.
$$
We have thus proved the following.

\begin{lemma}
\label{lem:SPX}
If $\Gamma \cong \SPX(r,s)$ with $r\ge 3$ and $1\le s\le r-1$, then
the inequality (\ref{eq:mug}) holds for every $g\in \Aut(\Gamma)$.
In particular, Theorem \ref{the:mainorbits} holds for the Split Praeger-Xu graphs.
\end{lemma}

This lemma, together with
 a recent result of Pablo Spiga and the first-named author of this paper \cite{fixicity},
 which bounds the number of vertices that can be fixed by a non-trivial automorphism in a cubic-vertex transitive graph, yields the following.
 
 \begin{corollary}
 \label{cor:SPX}
 Let $\Gamma$ be a cubic vertex-transitive graph on $n$ vertices
 admitting a non-identity automorphism fixing more than $\frac{n}{3}$ vertices of $\Gamma$.
 Then the inequality (\ref{eq:mug}) holds for every $g\in \Aut(\Gamma)$.
In particular, Theorem \ref{the:mainorbits} holds for such a graph $\Gamma$. 
 \end{corollary}

\begin{proof}
The claim can be easily checked for all vertex-transitive graphs on at most $20$ vertices
by consulting the census of all cubic vertex-transitive graphs of order at most $1280$ \cite{CVT} or by some other ad hoc techniques. We may thus assume that $n>20$.
By \cite[Theorem 1.2]{fixicity}, $\Gamma$ is then isomorphic to a Split Praeger-Xu graph $\SPX(r,s)$ with $r\ge 3$ and $s\le 2r/3$. The result then follows by Lemma~\ref{lem:SPX}.
\end{proof}

Let $\Gamma$ be a cubic vertex-transitive graph with $n>20$ vertices, and let $g \in \Aut(\Gamma)$. In view of Corollary \ref{cor:SPX}, we shall assume for the rest of the section that
every automorphism of $\Gamma$ fixes at most $n/3$ vertices of $\Gamma$.
We also let $G=\langle g \rangle$, $k = \frac{n}{o(g)} = \frac{n}{\ell(g)}$ and $m = o(g)$.
Recall that our aim is to show
that $G$ has at most $17k/6$ orbits on $\V(\Gamma)$.
By Lemma \ref{lem:orbitratio}, every orbit of $G$ has size $\frac{m}{i}$ for some $i \in \{1,2,3,4,6\}$. For $i \in \{1,2,3,4,6\}$, let $\Omega_i \subseteq \V(\Gamma)$ be the set of vertices contained in a $G$-orbit of size $\frac{m}{i}$ and let $N_i$ denote the number of $G$-orbits of size $\frac{m}{i}$. Note that $|\Omega_i| = \frac{m}{i}N_i$. 

By Lemma \ref{lem:pairs}, we have
\begin{align}
N_6 \leq N_2, \label{boundN6}\\ 
N_3 \leq N_1. \label{boundN3} 
\end{align}

Observe that  $N_1m + \frac{m}{3}N_3 = |\Omega_1| + |\Omega_3| \leq n = mk$. From this and inequality (\ref{boundN3}), we obtain

$N_3\frac{4m}{3} \leq N_1m + \frac{m}{3}N_3 \leq mk$
 and so
\begin{align}
\label{eq:N3}
N_3 \leq \frac{3}{4}k.
\end{align}

Moreover, we have $N_1 \leq k - N_3\frac{1}{3}$ and thus $N_1 + N_3 \leq k + \frac{2}{3}N_3$. It follows from (\ref{eq:N3}) that
\begin{align}
N_1 + N_3 \leq \frac{3}{2}k.
\label{eq:boundN1N3}
\end{align}

Clearly if $2 \nmid m$, then $G$ can only have orbits of size $m$ and $\frac{m}{3}$. That is, $\V(\Gamma) = \Omega_1 \cup \Omega_3$ and $\mu(g) = N_1 + N_3$. By inequality (\ref{eq:boundN1N3}) we have $\mu(g) \leq \frac{3}{2}k < \frac{17}{6}k$. Hence, we may assume that $2 \mid m$.

Now, consider the automorphism $g^{(m/2)}$ and observe that it fixes every orbit of size $\frac{m}{2}$ as well as every orbit of size $\frac{m}{4}$. That is, the set $\Omega_2 \cup \Omega_4$ is fixed point-wise by $g^{(m/2)}$. Then, since $g$ is not the identity, from Theorem \ref{the:fix}, we have $|\Omega_2 \cup \Omega_4| \leq \frac{n}{3}$. That is
\begin{equation}
\label{eq:omega24}
N_2\frac{m}{2} + N_4\frac{m}{4} \leq \frac{n}{3} = \frac{mk}{3},
\end{equation}
and thus
\begin{equation}
N_4 \leq \frac{4}{3}k - 2N_2.
\label{eq:boundN4}
\end{equation}

Now, $\mu(g)= N_1 + N_2 + N_3 + N_4 + N_6$, but $N_4 \leq \frac{4}{3}k - 2N_2$ and $N_6 \leq N_2$ by inequalities (\ref{eq:boundN4}) and (\ref{boundN6}). Then 
\begin{align*}
\mu(g) \leq N_1 + N_2 + N_3 + \left(\frac{4}{3}k - 2N_2\right) + N_2 = N_1 + N_3 + \frac{4}{3}k = \frac{17}{6}k,
\end{align*}
where the last equality follows from (\ref{eq:boundN1N3}). This completes the proof of Theorem \ref{the:mainorbits}.

We would like to point out that the the constant $\frac{17}{6}$ can featuring in Theorem \ref{the:mainorbits} could most likely be substituted with the integer $2$, however,
as following two extreme examples show, the ration $2$ is the best that one can hope for.
First, consider the case where $G:=\langle g \rangle$ is transitive (that is, when the graph is a circulant). Then $G$ has a single orbit of length $n:=|\V(\Gamma)|$, $n/o(g) =1$ and thus $\mu(g) = 1 = 2n/o(g)-1$. On the other side of the spectrum,
the split Praeger-Xu graphs $\SPX(n/2,1)$ admit a group of automorphisms of order $2$ having $n-2$ orbits on $\V(\Gamma)$, two of which have size $2$ (while all others have contain a single vertex). Here the number of orbits equals $2n/o(g)-2$. We suspect these are two of the most extreme case and make the following conjecture.     

\begin{conjecture}
\label{conjecture}
Let $\Gamma$ be a cubic vertex-transitive graph of order $n$, and let $g \in \Aut(\Gamma)$. If $k = \frac{n}{\ell(g)}$ then $\mu(g) \leq 2k - 1$.
\end{conjecture}

Now, let us get our attention to Theorem \ref{the:regorbitsratio}. Again, it can be easily verified in the census of cubic vertex-transitive graphs that the theorem holds for all cubic vertex-transitive graphs of order smaller than $20$, with the exception of $K_{3,3}$ and the split Praeger-Xu graphs. Moreover, by assumption, every non-trivial automorphism of $\Gamma$ fixes at most $n/3$ vertices, and thus $\Gamma$ is not a split Praeger-Xu graph. Then, to prove Theorem \ref{the:regorbitsratio}, it suffices to show that at least $\frac{5}{12}n$ vertices of $\Gamma$ lie on a regular orbit of $g$. That is, we need to show that $|\Omega_1| \geq \frac{5}{12}n$. Recall that $\Omega_6 = \frac{m}{6} N_6$ and $N_6 \leq N_2$. Moreover, by inequality( \ref{eq:omega24}), we have $N_2 \leq \frac{2}{3}k$ and thus
\begin{align}
\label{eq:omega6}
|\Omega_6| =  \frac{m}{6}N_6 \leq \frac{m}{6}N_2 \leq \frac{m}{6}\cdot\frac{2}{3}k = \frac{1}{9}mk = \frac{1}{9}n.
\end{align}

Now, $n = |\Omega_1| + |\Omega_2| +|\Omega_3| + |\Omega_4| + |\Omega_6|$ but $|\Omega_2| + |\Omega_4| \leq \frac{1}{3}n$ and $|\Omega_6| \leq \frac{1}{9}n$ by inequalities (\ref{eq:omega24}) and (\ref{eq:omega6}), respectively. Therefore

\begin{equation}
\label{eq:omega1}
|\Omega_1| = n - (|\Omega_2| + |\Omega_4|) - |\Omega_6| - |\Omega_3| \geq n - \frac{1}{3}n - \frac{1}{9}n - |\Omega_3| = \frac{5}{9}n - |\Omega_3|
\end{equation} 
 
Since $N_3 \leq N_1$ we have $\Omega_3 = \frac{m}{3} N_3 \leq \frac{m}{3} N_1 = \frac{1}{3}\Omega_1$. From this and (\ref{eq:omega1}), we have $|\Omega_1| \geq \frac{5}{9}n - |\Omega_3| \geq \frac{5}{9}n - \frac{1}{3} |\Omega_1|$. We conclude that $\Omega_1 \geq \frac{5}{12}$, thus proving Theorem \ref{the:regorbitsratio}.

As with the bound given in Theorem \ref{the:mainorbits}, here too we suspect that the ration $12/5$ appearing in Theorem~\ref{the:regorbitsratio} can be improved. On the other hand,
consider the following family of graphs. For an even integer $r > 0$, let $\Psi(r)$ be the graph with vertex set $\ZZ_r \times \ZZ_3$ and with edges of the form $(i,j)(i-1,j)$, $(i,j)(i+1,j+1)$ and $(i,j)(i+1,j+2)$  for all even $i \in \ZZ_r$. The graph $\Psi(r)$ is a cubic vertex-transitive graph of order $3r$ (for more details see \cite[Section 2.4]{girth6VT} where the graph $\Psi(r)$ is called $\Sigma_r$). Observe that the permutation $g$ that interchanges $(i,0)$ with $(i,1)$ while fixing $(i,2)$ for all $i \in \ZZ_r$ is an automorphism of $\Psi(r)$. The regular orbits of $g$ are the sets of the form $\{(i,0),(i,1)\}$. It follows that two thirds of the total number of vertices of $\Psi(r)$ lie on a regular orbit of $g$. We believe that, excluding the family of split Praeger-Xu graphs, this is an extreme case.

\begin{conjecture}
Let $\Gamma$ be a cubic vertex-transitive graph of order $n$ not isomorphic to $K_{3,3}$ or a split Praeger-Xu graph, and let $g \in \Aut(\Gamma)$. Then at least $\frac{2}{3}n$ vertices lie on a regular orbit of $g$.
\end{conjecture}

\section{The order of automorphisms of locally-semiprimitive graphs}
\label{sec:TW}


Let $\Gamma$ be a graph, let $G\le\Aut(\Gamma)$ and let $v\in \V(\Gamma)$. The permutation group induced by the action of the stabiliser $G_v$ on the neighbourhood
$\Gamma(v)$ of the vertex $v$ will be denoted $G_v^{\Gamma(v)}$. 
Observe that if $G$ acts transitively on $\V(\Gamma)$, then up to permutation isomorphism
the group $G_v^{\Gamma(v)}$ is independent of the choice of $v$.
In this case, if $L$ is an arbitrary permutation group permutation isomorphic to $G_v$.
we say that $\Gamma$ is of local $G$-action $L$; if $G=\Aut(\Gamma)$, then the prefix $G$ can be omitted.

A transitive permutation group $G\le \Sym(\Omega)$ is called semiregular provided that
$G_\omega = 1$ for every $\omega \in \Omega$. Furthermore,
following \cite{BerMar},
we call a transitive permutation group $G\le \Sym(\Omega)$ {\em semiprimitive} provided that
every normal subgroup $N$ of $G$ is either transitive or semiregular. Observe that every primitive, as well as every quasiprimitive
permutation group is semiprimitive. For a graph $\Gamma$ and $G\le \Aut(\Gamma)$ we
say that $\Gamma$ is {\em locally $G$-semiprimitive} whenever  the permutation group $G_v^{\Gamma(v)}$,  induced by the action of the vertex-stabiliser $G_v$ on the neighbourhood $\Gamma(v)$, is semiprimitive (or equivalently, the  local $G$-action of
$\Gamma$ is semiprimitive).

A still unresolved conjecture of Richard Weiss \cite{WeissConj} states that for every valence
$d$ there exists a constant $c_d$ such that for every connected $G$-arc-transitive locally $G$-primitive graph $\Gamma$ the order of the vertex-stabiliser is bounded by $c_d$, and in particular, $\mso(\Gamma) \le c_d$ and $\meo(\Gamma) \le c_d |\V(\Gamma)|$.
 Even though several deep partial results were proved, this conjecture is still wide open.
Weiss' conjecture was strengthened first by Cheryl Praeger
\cite{PraConj}, who relaxed the condition of local primitivity to local quasi-primitivity, and then by Spiga, Verret and the the first-named author of this paper \cite{PSV}, who relaxed the condition
to local semiprimitivity.

A starting point to most attempts to prove Weiss's conjecture is the so-called Thompson-Wielandt theorem, whose variant for locally semiprimitive graphs was proved in \cite{SpigaTW}.

\begin{theorem} \cite[Corollary 3]{SpigaTW}
\label{the:SpigaTW}
Let $\Gamma$ be a connected $G$-arc-transitive locally $G$-semiprimitive graph,
let $\{u,v\}$ be an edge of $\Gamma$ and let $G_{uv}^{[1]}$ be the point-wise stabiliser
of all the vertices at distance $1$ from $u$ or $v$. Then $G_{uv}^{[1]}$ is a $p$-group.
\end{theorem}

 Theorem~\ref{the:WieThomp} now easily follows from this result and Corollary~\ref{cor:TW} in the following way. Let connected $G$-arc-transitive locally $G$-semiprimitive graph of valence $d$ and let $\{u,v\}$ be an edge of $\Gamma$. Since $G_{uv}^{[1]}$ is the kernel of the action
 of the arc-stabiliser $G_{uv}$ on the set $X:=(\Gamma(u) \cup \Gamma(v)) \setminus \{u,v\}$ of cardinality at most $2d-2$. Hence $|G_v| = d |G_{uv}| \le d (2d-2)! |G_{uv}^{[1]}|$.
Since, by Theorem~\ref{the:SpigaTW}, $G_{uv}^{[1]}$ is a $p$-group, 
 Corollary~\ref{cor:TW} now implies that $o(g) \le d (2d-2)! \ell(g)$ for every $g\in G$.
 We have thus shown that $\meo(G) \le c_d |\V(\Gamma)|$, where
 $c_d = d (2d-2)!$. This completes the proof of  Theorem~\ref{the:WieThomp}.

\section{The order of automorphisms of quartic vertex-transitive graphs}
\label{sec:quartic}

In this section we prove Theorem~\ref{the:34} for the case of quartic graphs. That is,
we prove that $\meo(\Gamma) \le 9 |\V(\Gamma)|$ holds for every
finite connected vertex-transitive graph of valence $4$.
As we shall see, the proof quickly reduces to proving a bound on the exponent
of a Sylow $3$-subgroup of a vertex-stabiliser in a finite connected $6$-valent arc-transitive graph.  We thus begin by proving the following result, which is a generalisation of~\cite[Theorem~4.9]{GenLost}.

\begin{proposition}\label{top}
Let $L$ be a permutation group on $\Omega$, $p$ be a prime and $H$ a  $p$-subgroup of $L$. Suppose that there exist $x,y\in\Omega$  such that 
\begin{itemize}
\item $H=\langle H_x,H_y\rangle$, 
\item $x^{H}\cup y^{H}=\Omega$, and
\item $|H:H_x|=|H:H_y|=p$.
\end{itemize}
Let  $\Gamma$ be a connected $G$-vertex-transitive and $G$-edge-transitive graph with local $G$-action $L$, let $v$ be a vertex of $\Gamma$ and identify $G_v^{\Gamma(v)}$ with $L$. If $S$ is a $p$-subgroup of $G_v$ that maps to $H$ along the natural projection $G_v\to L$, then $S$ has the following properties:
\begin{enumerate}
\item $S$ has nilpotency class at most $3$;\label{bonbonbon}
\item $S$ contains an elementary abelian $p$-subgroup of order at least  $|S|^{2/3}$; \label{bonbon}
\item $|\Zent {S}|^3\geq |S|$; \label{boubou}
\item $S$ has exponent at most $p^2$.
\end{enumerate}
\end{proposition}
\begin{proof}
Let $u$ and $w$ be the neighbours of $v$ corresponding to $x$ and $y$ under the identification of $\Gamma(v)$ with $\Omega$.

We show that the arcs $(u,v)$ and $(v,w)$ are in the same $G$-orbit. We argue by contradiction and we suppose that this is not the case. Since $\Gamma$ is $G$-edge-transitive, it follows that $(u,v)$ is in the same $G$-orbit as $(w,v)$. This implies that $u$ and $w$ are in the same $G_v$-orbit and hence $x$ and $y$ are in the same $L$-orbit. This implies that $L$ is transitive, so  $\Gamma$ is $G$-arc-transitive and hence $(u,v)$ and $(v,w)$ are in the same $G$-orbit, which is a contradiction. 

Let $\phi\in G$ such that $(u,v)^\phi=(v,w)$. We show that $\langle S,\phi \rangle$ is transitive on $\V(\Gamma)$. For $i\in\ZZ$, let $v_i=v^{\phi^i}$ and let $S_i=S^{\phi^i}$. Note that $(v_{-1},v_0,v_1)=(u,v,w)$ and hence $\Gamma(v_0)=(v_{-1})^{S_0}\cup (v_{1})^{S_0}$. Conjugating by $\phi^i$, we obtain that $\Gamma(v_i)=(v_{i-1})^{S_i}\cup (v_{i+1})^{S_i}$ for every $i\in\ZZ$. Let $G^*=\langle S_i\mid i\in\ZZ\rangle$ and let $X=v^{\langle \phi\rangle}=\{v_i\mid i\in\ZZ\}$. Note that $G^*\leq \langle S,\phi \rangle$, and  hence it suffices to show that $X^{G^*}=\V(\Gamma)$. By contradiction, suppose that there exists a vertex not in $X^{G^*}$ and choose one with minimum distance to $X$. Call this vertex $\alpha$ and let $(p_0,\ldots,p_{n-1},p_n)$ be a shortest path from $\alpha$ to a vertex of $X$. In particular, $p_0=\alpha$ and $p_n=v_i$ for some $i\in\ZZ$.  Since $\Gamma(v_i)=(v_{i-1})^{S_i}\cup (v_{i+1})^{S_i}$, there exists $\sigma\in S_i\leq G^*$ such that $(p_{n-1})^\sigma\in\{v_{i-1},v_{i+1}\}\subseteq X$. Since $\alpha$ is not in $S^{G^*}$, neither is $\alpha^\sigma$, but $\alpha^\sigma$ is closer to $X$ than $\alpha$ is, which is a contradiction.

From now on, we follow the notation of~\cite{Currano} and~\cite{Glaub} as closely as possible. Let $P=S$, let $R=S_u$ and let $Q=S_w$. Note that $R^\phi=Q$, and $R$ and $Q$ both have index $p$ in $P$. 

Let $N$ be the subgroup of $P$ generated by all the subgroups of $R$ that are normalised by $\phi$. By~\cite[Proposition~2.1]{Glaub}, $N$ is normal in $P$. By definition, $N$ is normalised by $\phi$ and hence $N$ is normalised by $\langle P,\phi \rangle$. On the other hand, we have shown that $\langle P,\phi \rangle$ is transitive on $\V(\Gamma)$.  Since $N\leq P\leq G_v$, it follows that $N=1$. This shows that condition (1.1) of~\cite{Currano} is satisfied.

 Let $|S|=p^t$ and let $u$, $v$ and $x_1,\ldots,x_t$ be as in~\cite[Theorem~1]{Currano} and let $E=\langle x_1,\ldots,x_u\rangle$. By~\cite[Lemma~2.2~(d,f,g)]{Currano}, $P$ has nilpotency class at most $3$, $\Zent{P}=\langle x_{v+1},\ldots,x_u\rangle$ and $[P,P]$ is elementary abelian. It follows from~\cite[Lemma~2.1]{Currano} that $P$ is generated by elements of order $p$ and that $|P|=p^t$, $|E|=p^u$ and $|\Zent P|=p^{u-v}$. It also follows from~\cite[Theorem~1~(1.2-1.5)]{Currano} that $v=t-u$, $u\geq \frac{2}{3}t$, $E\leq P$ and $E$ is elementary abelian which concludes the proof of (\ref{bonbon}). Since $v=t-u$, $|\Zent P|=p^{u-v}=p^{2u-t}\geq p^{t/3}=|P|^{1/3}$ and (\ref{boubou}) follows. Finally, $P$ has exponent at most $p^2$ by~\cite[Lemma~3.6]{GenLost}.
\end{proof}

\begin{lemma}\label{TwoOrbitsP}
Let $p$ be an odd prime and let $L$ be a permutation group of degree $2p$, either transitive or having two orbits of size $p$, and let $H$ be a Sylow $p$-subgroup of $L$. One of the following holds:
\begin{enumerate}
\item $H$ is semiregular of order $p$, or 
\item $H$ satisfies the first part of the hypothesis of Proposition~\ref{top}.
\end{enumerate}
\end{lemma}
\begin{proof}
Since $p$ is odd, $|H|$ divides $p^2$. If $|H|=p^2$, then $H$ must be isomorphic to $\C_p\times \C_p$ acting naturally with two orbits of size $p$ and it is easy to check that the first part of the hypothesis of Proposition~\ref{top} is satisfied. Since $p$ divides the size of an orbit of $L$, $H\neq 1$ hence $|H|=p$. 

If $L$ is primitive then, a well known consequence of the classification of finite simple groups states that
 either $\Alt(2p)\leq L$ or $p=5$ and $\Alt(5)\leq L\leq\Sym(5)$. In the former case, $|H|=p^2$ while in the later case, $p$ does not divides $|L_x|$ hence $H$ is semiregular. If $L$ admits a system of $p$ blocks of size $2$, then again $p$ does not divides $|L_x|$. If $L$ admits a system of $2$ blocks of size $p$, then $H$ is contained in the setwise stabiliser of the blocks, which has two orbits of size $p$. So we may assume that $L$ has two orbits of size $p$.

We may assume that $H$ is not semiregular, and thus $|H_x|=p$ for some $x\in\Omega$. Since $|H|=p$, $H$ fixes $x$, so $H$ fixes $p$ points and has one orbit of size $p$, with representative $y$, say. Now, $x$ and $y$ must be representatives of the two orbits of $L$ and there must be another Sylow subgroup $H'$ that is transitive on $x^L$ and, since it is conjugate to $H$, it must fix $y^L$ pointwise. Now, $|\langle H,H'\rangle|=p^2$, which is a contradiction.
\end{proof}

Note that the hypothesis that $p$ is odd is necessary as if $p=2$, then $L=H=\D_4$ in its natural action is a counterexample. (It is neither semiregular nor generated by point-stabilisers.)

\begin{corollary}\label{cor:Gabriel}
Let $p$ be an odd prime and let $\Gamma$ be a connected $G$-vertex-transitive  and $G$-edge-transitive graph of valency $2p$, and let $S$ be a Sylow $p$-subgroup of a vertex-stabiliser $G_v$. Then $S$ has the properties at the end of Proposition~\ref{top}.
\end{corollary}
\begin{proof}
Let $L$ be the local action at $v$. This is a permutation group of degree $2p$. It is either transitive, or has two orbits of size $p$. Let $H$ be the projection of $S$ onto $L$. Note that $H$ is a Sylow $p$-subgroup of $L$. We apply Lemma~\ref{TwoOrbitsP} to conclude that $H$ is semiregular or $H$ satisfies the first part of Proposition~\ref{top}. If $H$ is semiregular, then $S$ is arc-semiregular and $|S|=p$ and clearly it satisfies all the properties. Otherwise, we apply Proposition~\ref{top}.
\end{proof}

Equipped with Corollary~\ref{cor:Gabriel}, we can now finish the proof of 
Theorem~\ref{the:34} for the case of quartic graphs. 
For the rest of the section, 
let $\Gamma$ be a finite connected vertex-transitive graph of valence $4$, let $G=\Aut(\Gamma)$ and let $v\in \V(\Gamma)$.
Then $G_v^{\Gamma(v)}$ is permutation isomorphic to one of the groups:
\begin{itemize}
\item[{\rm (1)}]
the doubly transitive permutation groups $\Sym(4)$, $\Alt(4)$ of degree $4$;
\item[{\rm (2)}]
the transitive groups $\D_4$, $\C_4$, $\C_2\times \C_2$ of degree $4$;
\item[{\rm (3)}]
$\{\id\}$,  $\C_2\times \C_2$ or $\C_3$ in their unique intransitive
faithful actions on $4$ points, or $\C_2$ either its action with two orbits of length $2$ or
in its action with three orbits, one of length $2$ two of length $1$;
\item[{\rm (4)}]
$\Sym(3)$ in its unique faithful action on $4$ points.
\end{itemize}

If (1) occurs, then $G$ acts transitively on the $2$-arcs of $\Gamma$ (where
a $2$-arc is a triple $(w,u,v)$ of distinct vertices such that $u$ is adjacent to both $w$ and $v$). Based on the results in \cite{Gard,WeissSmallVal} it was proved in \cite{Pot4} that
$G_v$ is then isomorphic to one of $7$ finite groups. In particular, it follows from
\cite[Theorem 4]{Pot4} that $\exp(G_v)$ divides  $2^3\cdot 3^2$.
It then follows from Corollary~\ref{cor:TW} that $o(g) \le 8 \ell(g)$ for every $g\in G$.

If either (2) or (3) occurs, then the fact that $G_v^{\Gamma(v)}$ is a $p$-group for $p=2$ or $p=3$, the connectivity of $\Gamma$ implies that $G_v$ is a $p$-group, and
Corollary~\ref{cor:TW} then yields that $o(g) = \ell(g)$.

For the rest of the section we will assume that (4) occurs. Then every vertex $v\in \V(\Gamma)$ has a unique neighbour $v'$ which is fixed by every automorphism in $G_v$.
Observe that $v'' = v$ for every $v\in \V(\Gamma)$ and that the set
$M=\{ \{v,v'\} : v\in \V(\Gamma)\}$ forms a complete matching of $\Gamma$
invariant under the action of $G$. 

Think of the edges in $M$ red and the edges outside $M$ blue. Similarly, call
an arc blue or red if its underlying edge is blue or red.
Observe that $G$ has two orbits on the arc-set of $\Gamma$, one consisting
of all blue arcs and and of all red arcs.

Let $\Lambda$ be the graph with vertex-set $M$ and with two
red edges $vv'$, $uu' \in M$ adjacent in $\Lambda$ whenever
one of $v,v'$ is adjacent to one of $u,u'$. Clearly, every element of $G$ (in its action on $M$) induces an automorphism of $\Lambda$. Let $K$ denote the kernel of the
action of $G$ on $M$. Then $G/K$ is a vertex-transitive group of automorphisms
of $\Lambda$. 

Let $e:=vv'$ be a red edge of $\Gamma$, let $a,b,c$ be the three neighbours of $v$ distinct from $v'$ and let $x,y,z$ be the three neighbours of $v'$ distinct from $v$.
Then the neighbourhood of $e$ in $\Lambda$ is $\{aa', bb',cc',dd',xx',yy',zz'\}$.
Observe also that the stabiliser $G_e$ acts transitively on the set $\{a,b,c,x,y,z\}$,
and thus also on the neighbourhood of $e$ in $\Lambda$. In particular, $G/K$
is not only vertex-transitive group of automorphism of $\Lambda$ but in fact arc-transitive.

Since $a,b,c$ are pairwise distinct, so are $aa'$, $bb'$ and $cc'$.
The valence of $\Lambda$ is thus at most $6$ and at least $3$.
Observe also that since $G_v^{\Gamma(v)} \cong \Sym(3)$ and since $G_v$,
 the vertex-stabiliser $G_v$ contains an element $g$ acting on $\{a,b,c\}$ as the permutation $(a\, b\, c)$ and on $\{x,y,z\}$ as $(x\, y\, z)$. 
 
 Suppose first that the valence of $\Lambda$ is less than $6$. Then the red edge $aa'$ equals one of the edges $xx'$, $yy'$ or $zz'$, say $xx'$. By applying the automorphism $g$ twice,  then see that $yy'=bb'$ and $zz'=cc'$, implying that the valence of $\Lambda$ is $3$ in this case.  In particular, $\Lambda$ is a cubic $G/K$-arc-transitive graph.
  By \cite{tutte}, the order of the vertex-stabiliser $(G/K)_e$ is then
 of the form $3\cdot 2^{s-1}$ for some $s\le 5$. 
 Let us now consider the kernel $K$.
  Since $G$ acts transitively on the arcs of $\Lambda$, the subgraph $B$ of $\Gamma$ induced by the blue edges between two adjacent distinct red edges $uu'$ and $ww'$ is independent of the choice of $uu'$ and $ww'$ and admits an automorphism swapping the pair $\{u,u'\}$ with the pair $\{w,w'\}$ . Since there are precisely
 six blue edges adjacent to any given red edge 
 it follows that $B$ consists of two edges and must thus be isomorphic to $2K_2$.
By the connectivity of $\Gamma$, this implies that every element of $K$ either
fixes each pair $\{w,w'\}$, $w\in \V(\Gamma)$ point-wise or swaps the two vertices
in each such pair. In particular, the order of $K$ is at most.
This implies that the order of $G_e$ is at most twice the order of 
$(G/K)_e$ and thus equal to $3\cdot 2^{s}$ for some $s\le 5$.
Since $G_v$ is of index $2$ in $G_e$, Corollary~\ref{cor:TW} implies that
$o(g) \le 3 \ell(g)$ for every $g\in G$.

We are thus left with the case where the valence of $\Lambda$ is $6$.
 Observe that then there is at most one blue edge
between every two red edges, implying that  $K=1$ and $G \le \Aut(\Lambda)$.
Moreover, the group $G_e^\Lambda(e)$ is permutation isomorphic to the
group induced by the action of $G_e$ on the vertices $\{a,b,c,x,y,z\}$.
Observe that the latter group imprimitive with $\{a,b,c\}$ being a block 
of imprimitivity, implying that its order is a a divisor of $2|\Sym(3)| = 12$.
The connectivity of $\Lambda$ then implies that $G_e$ is a  $\{2,3\}$-group.
On the other hand, by Corollary~\ref{cor:Gabriel},
the Sylow $3$-subgroup has exponent at most $9$. In particular,
$\exp(G_e)$ divides $9 \cdot 2^\alpha$ with $\alpha$ a positive integer.
Since $G_v$ has index $2$ in $G_e$, Corollary~\ref{cor:TW} then implies
that $o(g) \le 9 \ell(g)$ holds for every $g\in G$. This completes the proof of
Theorem~\ref{the:34}.


\begin{thebibliography}{SK}

\bibitem{BerMar} \'A.\ Bereczky, A.\ Mar\'oti, On groups with every normal subgroup transitive or semi-regular, {\em J.\ Algebra} {\bf 319} (2008), 1733--1751.

\bibitem{bondy} J. A. Bondy, U. S. R. Murty, Graph Theory with Applications. New York: North Holland, p. 244, 1976.

\bibitem{CamSheeSpi} P.\ Cameron, J.\ Sheehan, P.\ Spiga, Semiregular automorphisms of vertex-transitive cubic graphs,  {\em European J.\ Combin.} {\bf 27} (2006), 924--930.

\bibitem{ConDob} M.~Conder, P.~Dobcs\'{a}nyi, Trivalent symmetric graphs on up to 768 vertices, {\em J.\ Combin.\ Math.\ Combin.\ Comput}, \textbf{40} (2002), 41--63.  

\bibitem{Currano} J.~Currano, Finite $p$-groups with isomorphic subgroups, {\em Canad. J. Math.} \textbf{25} (1973), 1--13. 

\bibitem{djoko}  D.\ Djokovi\'{c},  A class of finite group-amalgams, {\em Proc.\ Amer.\ Math.\ Soc.} {\bf 80} (1980), 22--26. 

\bibitem{DjM} D.~Djokovi\'c, G.L.~Miller, Regular Groups of Automorphisms of Cubic Graphs, {\em J.\ Combin.\ Theory, Ser.\ B} {\bf 29} (1980) 195--230

\bibitem{sym45} B.\ Frelih, K.\ Kutnar, Classification of cubic symmetric tetracirculants and pentacirculants, {\em European J.\ Combin.} {\bf 34} (2013), 169--194.

\bibitem{Gard} A.\ Gardiner, Arc transitivity in graphs, {\em Quart.\ J.\ Math.\ Oxford Ser.} {\bf 24} (1973) 399--407.

\bibitem{regularcycles1} M. Giudici, C. Praeger, P. Spiga, Finite primitive permutation groups and regular cycles of their elements, {\em J. Algebra} {\bf 421} (2015), 27--55.

\bibitem{Glaub} G. Glauberman,  Isomorphic subgroups of finite $p$-groups. I, {\em Canad. J. Math.} \textbf{23}  (1971), 983--1022.

\bibitem{grosstuc} J.\ Gross, T.\ W.\ Tucker, {\em Topological graph theory},  Wiley-Interscience, New York, NY, (1987).

\bibitem{GueMorPreSpi15} S.\ Guest,  J.\ Morris, C.\ E.\ Praeger, P.\ Spiga, On the maximum orders of elements of finite almost simple groups and primitive permutation groups, {\em Trans.\ Amer.\ Math.\ Soc.} {\bf 367} (2015), 7665--7694.

\bibitem{GueMorPreSpi16} S.\ Guest,  J.\ Morris, C.\ E.\ Praeger, P.\ Spiga, Finite primitive permutation groups containing a permutation having at most four cycles, {\em J.\ Algebra} {\bf 454} (2016), 233--251.

\bibitem{GuestSpiga} S.\ Guest, P. Spiga, Finite primitive groups and regular orbits of group elements, {\em Trans.\ Amer.\ Math.\ Soc.} {\bf 369} (2017), 997--1024.

\bibitem{JajPotWilPX} R. Jajcay, P.\ Poto\v{c}nik, S.\ Wilson, The The Praeger-Xu Graphs: Cycle Structures, Maps and Semitransitive Orientations, {\em Acta Mathematica Universitatis Comenianae} {\bf 88} (2019), 269--291.

\bibitem{symtric} I. Kov\'{a}cs, K. Kutnar, D. Maru\v{s}i\v{c}, S. Wilson, Classification of cubic symmetric tricirculants, {\em Electronic J.\ Combin.} {\bf 19}(2) (2012), P24, 14 pages.

\bibitem{Landau} E.\ Landau, \"{U}ber die Maximalordnung der Permutationen gegebenen Grades,  {\em Arch.\ Math.\ Phys.} {\bf 5} (1903), 92--103.

\bibitem{Dragan} D.~Maru\v{s}i\v{c}, On vertex symmetric digraphs, {\em Discrete Math.}  {\bf 36} (1981), 69--81.

\bibitem{covers} A. Malni\v{c}, R. Nedela, M. Skoviera, Lifting graph automorphisms by voltage assignments, {\em Europ. J. Combin.} {\bf 21} (2000), 927--947.

\bibitem{elabcovers} A. Malni\v{c}, D.\ Maru\v{s}i\v{c}, P.\ Poto\v{c}nik, Elementary abelian covers of graphs, {\em J.\ Alg.\ Combin.} {\bf 20} (2004), 71--97.

\bibitem{SPX} J. Morris, P. Spiga, G. Verret, Semiregular automorphisms of cubic vertex-transitive graphs and the abelian normal quotient method, {\em Electro. J. Combin.} {\bf 20} (2015), P3.32. 

\bibitem{ATgirth} Y.-Q. Feng , R. Nedela, Symmetric cubic graphs of girth at most 7,  {\em Acta Univ. M. Belii Ser. Math.} {\bf 13} (2006), 33--35.

\bibitem{bic} T.\ Pisanski, A classification of cubic bicirculants,  {\em Discrete Math.} {\bf 307} (2007), 567--578.

\bibitem{Pot4} P.\ Poto\v{c}nik,   A list of 4-valent $2$-arc-transitive graphs and finite faithful amalgams of index $(4, 2)$,   {\em European.\ J.\ Combin.} {\bf 30} (2009), 1323--1336.
       
\bibitem{fixicity} P.\ Poto\v{c}nik, P.\ Spiga, On the number of fixed points of automorphisms of vertex-transitive graphs, to appear in {\em Combinatorica}; see also:  \url{https://arxiv.org/abs/1909.05456}.

\bibitem{PSV} P.~Poto\v{c}nik, P.~Spiga, G.~Verret, On graph-restrictive permutation groups, {\em J.\ Combin. Theory, Ser.\ B} {\bf 102} (2012), 820--831.

\bibitem{CVT} P.\ Poto\v{c}nik, P.\ Spiga, G.\ Verret, Cubic vertex-transitive graphs on up to 1280 vertices, {\em J.\ Symb.\ Comp.} {\bf 50} (2013), 465--477.

\bibitem{HAT} P.~Poto\v{c}nik, P.~Spiga, G.~Verret, A census of 4-valent half-arc-transitive graphs and arc-transitive digraphs of valence two, {\em Ars Math. Contemp.} {\bf 8} (2015), 133--148.

\bibitem{tricirc} P.\ Poto\v{c}nik, M. Toledo, Classification of cubic vertex-transitive tricirculants, {\em Ars Math.\ Contemp.} {\bf 18} (2020) 1--31.  

\bibitem{cycliccovers} P.\ Poto\v{c}nik, M. Toledo, Finite cubic graphs admitting a cyclic group of automorphisms with at most three orbits on vertices, {\em Discrete Math.} {\bf 344} (2021), Article 112195.

\bibitem{longorbit} P.\ Poto\v{c}nik, M. Toledo, Cubic vertex-transitive graphs admitting an automorphism with a long orbit, {\em in preparation}.

\bibitem{MPgc} P.~Poto\v cnik, M.~Toledo,  Generalised voltage graphs, to appear in {\em European J.\ Combin}.; see also: \url{arXiv:1910.08421}.
 
\bibitem{PotVerDigraphs} P.\ Poto\v{c}nik, G.\ Verret,On the vertex-stabiliser in arc-transitive digraphs, {\em   J.\ Combin.\ Theory Ser.\ B.} {\bf 100} (2010), 497--509.
      
\bibitem{signature} P.\ Poto\v{c}nik, J.\ Vidali, Girth-regular graphs, {\em Ars Math.\ Contemp.} {\bf 17} (2019) 249--368.

\bibitem{girth6VT} P.\ Poto\v{c}nik, J.\ Vidali, Cubic vertex-transitive graphs of girth 6, \url{arXiv:2005.01635}.

\bibitem{recipes} P.\ Poto\v{c}nik, S.\ Wilson, Recipes for edge-transitive tetravalent graphs, {\em The art of discrete and applied mathematics} (2020), \#P1.08.

\bibitem{PraDir} C.~E.~Praeger,  Highly arc transitive digraphs, {\em European J.\ Combin.} {\bf 10} (1989), 281--292.

\bibitem{PraConj}  C.\ E.\ Praeger, Finite quasiprimitive group actions on graphs and designs, in {\em Groups--Korea '98}, de Gruyter (2000), 319--331.

\bibitem{SPX2} C.\ E.\ Praeger, M. Y. Xu. A Characterization of a Class of Symmetric Graphs of Twice Prime Valency. {\em European J. Combin.} {\bf 10} (1989), 91--102.

\bibitem{regularcycles2} J. Siemons, A. Zalesskii, Intersections of matrix algebras and permutation representations of ${\rm PSL}(n,q)$ {\em J. Algebra}, {\bf 226} (2000), 451--478.

\bibitem{regularcycles3} J. Siemons, A. Zalesskii, Regular orbits of cyclic subgroups in permutation representations of certain simple groups {\em J. Algebra}, {\bf 256} (2002), 611--625.

\bibitem{SpigaCameronConj} P.\ Spiga, Semiregular elements in cubic vertex-transitive graphs and the restricted Burnside problem, {\em Math.\ Proc.\ Cambridge Philos.\ Soc.} {\bf 157} (2014),  45--61.

\bibitem{PabloOral} P.\ Spiga, {\em Personal communication}.

\bibitem{SpigaTW} P.\ Spiga,  Two local conditions on the vertex stabiliser of arc-transitive graphs and their effect on the Sylow subgroups. {\em J.\ Group Theory} {\bf 15} (2012),  23--35.

\bibitem{GenLost} P.~Spiga, G.~Verret, On the order of vertex-stabilisers in vertex-transitive graphs with local group $\C_p\times \C_p$ or $\C_p \wr \C_2$, {\em J. Algebra} {\bf 48} (2016), 174--209.

\bibitem{sage}  W.\thinspace{}A. Stein et~al., \emph{{S}ageMath, the {S}age {M}athematics {S}oftware {S}ystem ({V}ersion  8.3)}, The Sage Developers, 2018, {\tt [https://www.sagemath.org}].

\bibitem{tutte} W.\ T.\ Tutte,  A family of cubical graphs,  {\em Proc.\ Cambridge Philos.\ Soc.}\ {\bf 43} (1947), 459--474 .

\bibitem{VanBon} J.\ Van Bon, Thompson-Wielandt-like theorems revisited, {\em Bull.\ London Math.\ Soc.} {\bf 35} (2003), 30--36.

\bibitem{WeissConj} R.\ Weiss,  $s$-transitive graphs, {\em Colloq.\ Math.\ Soc.\ J\'anos Bolyai} {\bf 25} (1978), 827--847.

\bibitem{WeissSmallVal} R.\ Weiss, Presentation for $(G, s)$-transitive graphs of small valency, {\em Math.\ Proc.\ Philos.\ Soc.} {\bf 101} (1987), 7--20.

\bibitem{TW} H.\ Wielandt, {\em Subnormal subgroups and permutation groups (Lecture Notes)}, Ohio State University, Columbus, Ohio (1971).
\end{thebibliography}
\end{document}